\documentclass[11pt,  reqno]{amsart}

\usepackage{amsmath,amssymb,amscd,amsthm,amsxtra, esint}

\usepackage[implicit=true]{hyperref}

\usepackage{color}

\allowdisplaybreaks[2]

\definecolor{gr}{rgb}   {0.,   0.69,   0.23 }
\definecolor{bl}{rgb}   {0.,   0.5,   1. }
\definecolor{mg}{rgb}   {0.85,  0.,    0.85}
\definecolor{yl}{rgb}   {0.8,  0.7,   0.}
\definecolor{or}{rgb}  {0.7,0.2,0.2}

\newtheorem{theorem}{Theorem} [section]

\newtheorem{lemma}[theorem]{Lemma}
\newtheorem{proposition}[theorem]{Proposition}
\newtheorem{remark}[theorem]{Remark}

\newtheorem{definition}[theorem]{Definition}

\newcommand{\1}{\hspace{0.5mm}\text{I}\hspace{0.5mm}}
\newcommand{\II}{\text{I \hspace{-2.8mm} I} }

\newcommand{\noi}{\noindent}
\newcommand{\Z}{\mathbb{Z}}
\newcommand{\R}{\mathbb{R}}
\newcommand{\C}{\mathbb{C}}
\newcommand{\T}{\mathbb{T}}
\newcommand{\N}{\mathbb{N}}

\let\Re=\undefined\DeclareMathOperator*{\Re}{Re}
\let\Im=\undefined\DeclareMathOperator*{\Im}{Im}

\let\P= \undefined
\newcommand{\P}{\mathbf{P}}

\newcommand{\F}{\mathcal{F}}
\newcommand{\NN}{\mathcal{N}}
\newcommand{\RR}{\mathcal{R}}
\newcommand{\TT}{\mathcal{T}}

\newcommand{\EE}{\mathcal{E}}

\newcommand{\al}{\alpha}

\newcommand{\dl}{\delta}

\newcommand{\eps}{\varepsilon}

\newcommand{\g}{\gamma}
\newcommand{\G}{\Gamma}
\newcommand{\ld}{\lambda}

\newcommand{\s}{\sigma}
\newcommand{\ft}{\widehat}
\newcommand{\wt}{\widetilde}
\newcommand{\cj}{\overline}
\newcommand{\dx}{\partial_x}
\newcommand{\dt}{\partial_t}

\newcommand{\jb}[1]
{\langle #1 \rangle}

\renewcommand{\l}{\ell}

\newcommand{\les}{\lesssim}

\newcommand{\ind}{\mathbf 1}

\newcommand{\bn}{{\bf n}}

\newcommand{\GG}{\mathcal{G}}

\newcommand{\E}{\mathbb{E}}

\renewcommand{\o}{\omega}
\renewcommand{\O}{\Omega}

\numberwithin{equation}{section}
\numberwithin{theorem}{section}

\newtheorem*{ackno}{Acknowledgements}

\usepackage{tikz}

\usetikzlibrary{shapes.misc}
\usetikzlibrary{shapes.symbols}
\usetikzlibrary{shapes.geometric}
\tikzset{
	dot/.style={circle,fill=black,draw=black,inner sep=1pt,minimum size=0.5mm},
	>=stealth,
	}
\tikzset{
	ddot/.style={circle,fill=white,draw=black,inner sep=2pt,minimum size=0.8mm},
	>=stealth,
	}

\tikzset{decision/.style={ 
        draw,
        diamond,
        aspect=1.5
    }}

\tikzset{dia2/.style
={diamond,fill=white,draw=black,inner sep=0pt,minimum size=1mm},
	>=stealth,
	}

\tikzset{dia/.style
={star,fill=black,draw=black,inner sep=0pt,minimum size=1mm},
	>=stealth,
	}

\makeatletter
\def\DeclareSymbol#1#2#3{\expandafter\gdef\csname MH@symb@#1\endcsname{\tikz[baseline=#2,scale=0.15]{#3}}}
\def\<#1>{\csname MH@symb@#1\endcsname}
\makeatother

\DeclareSymbol{X}{-2.4}{\node[dot] {};}
\DeclareSymbol{1}{0}{\draw[white] (-.4,0) -- (.4,0); \draw (0,0)  -- (0,1.2) node[dot] {};}
\DeclareSymbol{2}{0}{\draw (-0.5,1.2) node[dot] {} -- (0,0) -- (0.5,1.2) node[dot] {};}

\DeclareSymbol{T0}{-2.7}
 { \draw (0,0) node[ddot]{};}

\DeclareSymbol{T1}{0}
 {  
 \draw (0,0) node[dot]{} -- (0,4) node[ddot] {}; 
 \draw (-3,0) node[dot] {} -- (0,4)node[ddot] {} -- (3,0) node[dot] {};
\draw[line width=1pt] (0, 4)node[ddot, label=above:$r_1$]{} 
.. controls(3,6) and (5,6) ..
(8, 4)node[ddot, label=above:$r_2$]{}
(4, 12) node[label ] {$\underline{j = 1}$};
 }

\DeclareSymbol{T2}{0}
 {\draw (0,0) node[dot]{} -- (0,4) node[ddot] {}; 
 \draw (-3,0) node[dot] {} -- (0,4)node[ddot] {} -- (3,0) node[dot] {};
 \draw (0,-4) node[dot]{} -- (0,0) node[dot] {}; 
 \draw (-3,-4) node[dot] {} -- (0,0)node[dot] {} -- (3,-4) node[dot] {};
\draw[line width=1pt] (0, 4)node[ddot, label=above:$r_1$]{} 
.. controls(3,6) and (5,6) ..
(8, 4)node[ddot, label=above:$r_2$]{}
(4, 12) node[label ] {$\underline{j = 2}$};
 }

\DeclareSymbol{T3}{0}
 {\draw (0,0) node[dot]{} -- (0,4) node[ddot] {}; 
 \draw (-3,0) node[dot] {} -- (0,4)node[ddot] {} -- (3,0) node[dot] {};
 \draw (0,-4) node[dot]{} -- (0,0) node[dot] {}; 
 \draw (-3,-4) node[dot] {} -- (0,0)node[dot] {} -- (3,-4) node[dot] {};
\draw (10,0) node[dot]{} -- (10,4) node[ddot] {}; 
 \draw (7,0) node[dot] {} -- (10,4)node[ddot] {} -- (13,0) node[dot] {};
\draw[line width=1pt] (0, 4)node[ddot, label=above:$r_1$]{} 
.. controls(4,6) and (6,6) ..
(10, 4)node[ddot, label=above:$r_2$]{}
(5, 12) node[label ] {$\underline{j = 3}$};
 }

\tikzstyle{dot1} = [ draw=  gray!00, 
 rectangle, rounded corners, fill=gray!00,  inner sep=1pt, inner ysep=3pt]

\tikzstyle{dot2} = [ draw=  black, 
ellipse, fill=gray!00,  inner sep=1pt, inner ysep=3pt]

\tikzstyle{dot3} = [ draw=  gray!00, 
ellipse, fill=gray!00,  inner sep=1pt, inner ysep=3pt]

\DeclareSymbol{T4}{0}
 {\draw (0,0) node[dot1]{$\phantom{n_2^{(1)} = n^{(2)}}$} -- (0,15) node[dot1] {$\phantom{n^{(1)}}$}; 
 \draw (-13,0) node[dot1] {$n_1^{(1)}$} -- (0,15)node[ddot] {$\phantom{n^{(1)}}$} 
 -- (13,0) node[dot1] {$n_3^{(1)}$};
 \draw (0,-15) node[dot1]{$n_2^{(2)}$} -- (0,0) node[dot1] {$\phantom{n_2^{(1)} = n^{(2)}}$}; 
 \draw (-13,-15) node[dot1] {$n_1^{(2)}$} -- (0,0)node[dot1] {$n_2^{(1)} = n^{(2)}$} -- (13,-15) node[dot1] {$n_3^{(2)}$};
\draw (40,0) node[dot1]{{$n_2^{(3)}$}} -- (40,15) node[dot3]  {$\phantom{n^{(1)} = n^{(3)}}$} ; 
 \draw (27,0) node[dot1] {$n_1^{(3)}$} -- (40,15)node[dot3]  {$\phantom{n^{(1)} = n^{(3)}}$}  -- (53,0) node[dot1] {$n_3^{(3)}$};
\draw[line width=1pt] (0, 15)node[ddot, label=above:$r_1$]{$n^{(1)}$} 
.. controls(10,23) and (27,23) ..
(40, 15)node[dot2, label=above:$r_2$] {$n^{(1)} = n^{(3)}$} ;
 }

\begin{document}

\baselineskip = 14pt

\title[Quasi-invariant measures for the cubic fourth order NLS]
{An optimal regularity result on the quasi-invariant Gaussian measures for the cubic fourth
order nonlinear Schr\"odinger equation}

\author[T.~Oh, P.~Sosoe, and  N.~Tzvetkov]
{Tadahiro Oh, Philippe Sosoe, and Nikolay Tzvetkov}

\address{
Tadahiro Oh\\ School of Mathematics\\
The University of Edinburgh\\
and The Maxwell Institute for the Mathematical Sciences\\
James Clerk Maxwell Building\\
The King's Buildings\\
Peter Guthrie Tait Road\\
Edinburgh\\ 
EH9 3FD\\
 United Kingdom}

\email{hiro.oh@ed.ac.uk}

\address{
Philippe Sosoe\\
Department of Mathematics\\
Cornell University\\ 
584 Malott Hall\\ Ithaca\\ New York 14853\\ USA
}

\email{psosoe@math.cornell.edu}

\address{
Nikolay Tzvetkov\\
Universit\'e de Cergy-Pontoise\\
 2, av.~Adolphe Chauvin\\
  95302 Cergy-Pontoise Cedex \\
  France}

\email{nikolay.tzvetkov@u-cergy.fr}



\begin{abstract}
We study the transport properties of the Gaussian measures
on Sobolev spaces under the dynamics of  the cubic fourth order  nonlinear Schr\"odinger equation
on the circle.
In particular, 
we establish an optimal regularity result
for quasi-invariance of  the mean-zero Gaussian measures on Sobolev spaces.
The main new ingredient is an improved energy estimate
established by performing an infinite iteration of normal form reductions
on the energy functional.
Furthermore, we show that the dispersion is essential
for such a quasi-invariance result by 
proving  non quasi-invariance
of  the Gaussian measures 
under the dynamics of  the dispersionless model.

\end{abstract}


\maketitle

\vspace*{-5mm}

\tableofcontents

\section{Introduction}
\label{SEC:1}

In this paper, we complete the study of the transport properties of Gaussian measures on Sobolev spaces for the cubic nonlinear Schr\"odinger equation (NLS) with quartic dispersion, 
initiated  by the first and third authors in \cite{OTz1}.

The question addressed in this work is motivated by a number of perspectives. In probability theory, absolute continuity properties for the pushforward of Gaussian measures under linear and nonlinear 
transformations have been studied extensively, starting with the classical work of Cameron-Martin; see \cite{CM, KU,  RA}. More generally, questions of absolute continuity of the distribution of solutions 
to differential and stochastic differential equations 
with respect to a given initial distribution or some chosen reference measure are also central to stochastic analysis. For example, close to the topic 
 of the current paper, 
see the work of Cruzeiro \cite{Cru1,Cru2}.
We also note a recent work 
\cite{NRBSS} 
establishing absolute continuity 
of the Gaussian measure associated to the complex Brownian bridge
on the circle
under certain gauge transformations.

On the other hand, in the analysis of partial differential equations (PDEs), 
Hamiltonian PDE dynamics with initial data distributed according to measures of Gibbs type have been studied intensively over the last two decades, starting with the work of Bourgain \cite{BO94, BO96}. 
See \cite{OTz1} for the references therein.
These Gibbs-type measures are constructed as weighted Gaussian measures
and are usually supported on Sobolev spaces of low regularity
with the exception of completely integrable Hamiltonian PDEs
such as the cubic NLS on the circle.
In the approach initiated by Bourgain and successfully 
applied to  many equations since then, 
invariance of such Gibbs-type measures 
under the flow of the equation 
has been established
by combining the Hamiltonian structure of suitable finite dimensional approximations,
in particular invariance of the finite dimensional Gibbs-type measures, 
with PDE approximation arguments.
Invariance of such weighted Gaussian measures
implies 
absolute continuity of the pushforward of 
the base Gaussian measures.
If we substitute the underlying measure with a different Gaussian measure, however, 
the question of absolute continuity becomes non-trivial. 
See also \cite{Bo-Gelfand}
for a related question by Gel'fand on building   a direct method to prove absolute continuity properties
without relying on invariant measures.

In \cite{TzBBM}, the third author initiated the study of transport properties of Gaussian measures 
under the flow of a Hamiltonian PDE, combining probabilistic and PDE techniques. 
The result proved there 
 for a specific Hamiltonian equation (the generalized BBM equation) went beyond general results 
 on the pushforwards of Gaussian measures by nonlinear transformations such as Ramer's \cite{RA}. 
 It was shown in \cite{TzBBM} that a key step to showing absolute continuity is to establish a smoothing effect on the nonlinear part. 
 In \cite{OTz1}, the first and third authors studied the transport of Gaussian measures 
 for the cubic NLS with quartic dispersion. 
 An additional difficulty compared to \cite{TzBBM} is the absence of explicit smoothing coming from the nonlinearity,
 thus requiring the use of {\it dispersion} in an explicit manner.
 In~\cite{OTz1}, such dispersion was manifested
 through the normal form method.
 In this paper, we improve the result in \cite{OTz1} to the optimal range of Sobolev exponents
 by pushing  the normal form method to the limit. 
Furthermore, we present a result showing that, in the absence of dispersion, the distribution of the solution of the resulting dispersionless equation is {\it not}
absolutely continuous  with respect to the Gaussian initial data for any non-zero time. 
This in particular establishes the necessity of dispersion 
for an absolute continuity property.
Since the linear equation is easily seen to leave the distribution of the Gaussian initial data invariant, this highlights that the question of transport properties for a Hamiltonian PDE is a probabilistic manifestation of the competition between the dispersion and the nonlinear part, familiar for the study of nonlinear dispersive equations.

\subsection{The equation} 

We consider the cubic fourth order nonlinear Schr\"odinger equation (4NLS) on the circle 
$\mathbb{T} = \R /(2\pi \Z)$:
\begin{equation}
\label{4NLS0}
\begin{cases}
i \dt u = \dx^4 u \pm |u|^2u\\
u|_{t=0}= u_0,
\end{cases}
  \qquad (x,t)\in \mathbb{T}\times \mathbb{R},
\end{equation}

\noi
where  $u$ is a complex-valued function 
on $\T\times \R$.
The equation \eqref{4NLS0}
is also called 
the biharmonic NLS
and it was studied in \cite{IK, Turitsyn} in the context of stability of solitons in magnetic materials.
See also \cite{Karpman, KS, BKS,  FIP}
for a more general
class of fourth order NLS:
\begin{align}
i \dt u =  \ld \dx^2 u + \mu  \dx^4 u \pm  |u|^{2}u. 
\label{4NLS}
\end{align}

The equation \eqref{4NLS0} is a Hamiltonian PDE with the conserved Hamiltonian:
\[H(u)=\frac{1}{2}\int_{\T}|\dx^2 u|^2 dx
\pm \frac{1}{4}\int_\T |u|^4 dx.\]

\noi
In addition to the Hamiltonian, the flow of the equation \eqref{4NLS0}
preserves the $L^2$-norm, or the so-called ``mass'':
\begin{equation*}
M(u) = \int_{\mathbb{T}}|u|^2dx.
\end{equation*}

\noi
This mass conservation law was used in \cite{OTz1} to prove the following sharp global well-posedness result.

\begin{proposition}
The cubic 4NLS  \eqref{4NLS0} is globally well-posed in $H^\s(\T)$ for $\s\ge 0$.
\end{proposition}

This global well-posedness result in $L^2(\T)$
is sharp in the sense that  the cubic 4NLS~\eqref{4NLS0} is ill-posed in negative Sobolev spaces
in the sense of non-existence of solutions.
See~\cite{GO, OTz1, OW}.

The defocusing/focusing nature of the equation \eqref{4NLS0} does not play any role
in the following.
Hence,  we assume that it is defocusing, i.e.~with the $+$ sign in \eqref{4NLS0}.

\subsection{Quasi-invariance of $\mu_s$}
Given $s > \frac 12$, 
we consider the mean-zero Gaussian measures $\mu_s$ on $L^2(\T)$ with covariance operator 
$2(\text{Id}-\dx^2)^{-s}$, formally written as 
\begin{align}
 d \mu_s 
  = Z_s^{-1} e^{-\frac 12 \| u\|_{H^s}^2} du 
  =  \prod_{n \in \Z} Z_{s, n}^{-1}e^{-\frac 12 \jb{n}^{2s} |\ft u_n|^2}   d\ft u_n .
\label{gauss0}
 \end{align}

\noi
As we see below, 
the Gaussian measure $\mu_s$ is not supported on $H^s(\T)$, i.e.~$\mu_s(H^s(\T)) = 0$,
and we need to work in a larger space.
See \eqref{reg}.
This is due to the infinite dimensionality of the problem.

The covariance operator is diagonalized by the Fourier basis on $\mathbb{T}$ 
and the Gaussian measure $\mu_s$ defined above is in fact the induced probability measure
under the map\footnote{Henceforth, we drop the harmless factor of $2\pi$,
if it does not play an important role.}
\begin{align}
 \o \in \O \longmapsto u^\o(x) = u(x; \o) = \sum_{n \in \Z} \frac{g_n(\o)}{\jb{n}^s}e^{inx}, 
\label{gauss1}
 \end{align}

\noi
where  $\jb{\,\cdot\,} = (1+|\cdot|^2)^\frac{1}{2}$
and 
$\{ g_n \}_{n \in \Z}$ is a sequence of independent standard complex-valued 
Gaussian random variables on a probability space $(\O, \mathcal{F}, P)$, i.e.~$\text{Var}(g_n) = 2$.
From this random Fourier series representation, 
it is easy to see that 
 $u^\o$ in \eqref{gauss1} lies   in $H^{\s}(\T)$ 
  almost surely
 if and only if 
 \begin{align}
 \s < s -\frac 12.
\label{reg}
 \end{align}
 
 \noi
Lastly, note that, for the same range of $\s$, 
the triplet $(H^s, H^\s, \mu_s)$ forms an abstract Wiener space.
See \cite{GROSS, Kuo2}.

In the following, we continue to study the transport
property of the Gaussian measure $\mu_s$ under the dynamics of the cubic 4NLS \eqref{4NLS0}.
Before proceeding further, recall the following definition of quasi-invariant measures;
given a measure space $(X,\mu)$, 
we say that the measure  is {\it quasi-invariant} under a measurable transformation $T:X\to X$ if  $\mu$ and the pushforward of $\mu$ under $T$, defined by $T_*\mu = \mu\circ T^{-1}$, are 
equivalent, i.e.~mutually absolutely continuous with respect to each other.

Our first result improves the quasi-invariance result in \cite{OTz1} to the optimal range of Sobolev exponents.

\begin{theorem}\label{THM:1}
Let $s > \frac 12$.  Then, the Gaussian measure $\mu_s$ is quasi-invariant under the flow of 
the  cubic 4NLS
 \eqref{4NLS0}. 

\end{theorem}

Theorem \ref{THM:1} improves the main result 
in \cite{OTz1}, 
where the first and the third authors proved quasi-invariance of $\mu_s$
under \eqref{4NLS0}
for $s > \frac 34$.
Moreover, the regularity $s > \frac 12$ is optimal
since when $ s= \frac 12$, the Gaussian measure $\mu_s$ is
supported in $H^{-\eps}(\T)\setminus L^2(\T)$, while 
the cubic 4NLS \eqref{4NLS0} is ill-posed in negative Sobolev spaces
in the sense of non-existence of solutions.

As shown in \cite{TzBBM}, to prove quasi-invariance of $\mu_s$, it is essential to exhibit 
a smoothing of the nonlinear part of the equation. 
This can be understood at an intuitive level by an analogy to the Cameron-Martin theorem: 
the Gaussian measures $\mu_s$ are quasi-invariant under translations by fixed vectors in their respective Cameron Martin spaces $H^s(\T)$. 
Since a typical element under $\mu_s$ lies in $H^\sigma(\T)$,  $\sigma<s-\frac 12$, one needs to show that the nonlinear part represents a perturbation which is smoother in the Sobolev regularity. 
The Cameron-Martin theorem applies only to translation by fixed vectors, but 
Ramer's quasi-invariance result \cite{RA} applies to a more general nonlinear transformation
on an abstract Wiener space, although it requires the translations to be more regular. 
This was applied in \cite{TzBBM} and~\cite{OTz1}, 
where it was noted that a direct application of Ramer's result yields a suboptimal range on $s$. 
In \cite{OTz1}, 
we applied the normal form reduction
to the equation and exhibited $(1+\eps)$-smoothing on the nonlinearity
when $ s> 1$.
We then 
proved quasi-invariance of $\mu_s$
by invoking Ramer's result.
When $\frac 34< s \leq 1$, 
we followed the general approach 
introduced by the third author  
in the context of the (generalized) BBM equation \cite{TzBBM}.
This strategy combines an energy estimate with the analysis of the evolution of truncated measures. 
As in \cite{OTz1}, showing a smoothing of the nonlinear part for \eqref{4NLS0} requires normal form reductions. The main improvements over \cite{OTz1} here comes from a more refined implementation of the normal form reductions, inspired by \cite{GKO}.

In the following, we first describe a rough idea behind this method introduced in \cite{TzBBM}.
Let $\Phi(t)$ denote the solution map of \eqref{4NLS0}
sending initial data $u_0$ to the solution $u(t)$ at time $t \in \R$.
Suppose that we have a measurable set $A \subset L^2(\T)$ with  $\mu_s(A) = 0$.
Fix non-zero $t \in \R$.
In order to prove quasi-invariance of $\mu_s$,  we would like to prove $\mu_s(\Phi(t)(A)) = 0$.\footnote{By time reversibility,
this would also yield $\Phi(t)_*\mu_s(A) = \mu_s(\Phi(-t) (A)) = 0$.}
The main idea is to establish the following two properties:

\smallskip

\begin{itemize}
\item[(i)] Energy estimate (with smoothing):
\begin{align}
\frac{d}{dt} \|\Phi(t) (u)\|_{H^s}^2
\, \text{``}\leq\text{''} \, C(\| u  \|_{L^2})
\| \Phi(t) u \|_{H^{s - \frac 12 - \eps}}^{2-\theta}
\label{P0}
\end{align}

\noi
for some $\theta > 0$,\footnote{In \cite{OTz2}, 
the first and third authors recently proved quasi-invariance of $\mu_s \otimes \mu_{s-1}$
on $(u, \dt u)$
under the dynamics of the two-dimensional cubic nonlinear wave equation (NLW),
where 
they showed that 
even when $\theta = 0$, 
we can still  apply Yudovich's argument in the limiting case
and establish a desired  estimate of the form \eqref{Yudo}.
This was crucial in proving quasi-invariance of $\mu_s \otimes \mu_{s-1}$
under the cubic NLW on $\T^2$.}

\medskip

\item[(ii)] A change-of-variable formula:
\begin{align*}
\mu_s(\Phi(t)(A)) 
& =  Z_{s}^{-1}\int_{\Phi(t)(A)}
e^{-\frac 12 \|u\|_{H^s}^2} du \
 \text{``}=\text{''} \
Z_s^{-1}\int_{A} 
e^{-\frac 12  \|\Phi(t)(u)\|_{H^s}^2} du. 
\end{align*}

\end{itemize}

Step (i) is an example of local analysis, studying a trajectory of a single solution,
while Step (ii) is an example of global analysis on the phase space. 
Combining (i) and (ii), we 
can study the evolution of $\mu_s(\Phi(t) A)$
by estimating 
$\frac d{dt} \mu_s(\Phi(t) (A))$.
In particular, by applying Yudovich's argument \cite{Y}, we obtain\footnote{Compare \eqref{Yudo}
with a much stronger estimate in Lemma \ref{LEM:meas4}.}
\begin{align}
\mu_s (\Phi(t)(A)) \leq C(t,  \dl) \big(\mu_s ( A)\big)^{1-\delta}
\label{Yudo}
\end{align}

\noi
for any $\dl > 0$.
In particular, if $\mu_s(A) = 0$, then 
we would have $\mu_s(\Phi(t)(A)) = 0$.
See \cite{TzBBM, OTz1, OTzXEDP}
for details.

As the quotation marks indicate, 
both (i) and (ii) are not quite true as they are stated above.
In \cite{OTz1}, 
we first  performed  two transformations to \eqref{4NLS0}
and transformed the equation  into  the following renormalized equation:
\begin{align}
\dt \ft v_n 
& = -i \sum_{\substack{n = n_1 - n_2 + n_3\\n \ne n_1, n_3}} 
e^{-i \phi(\bar n) t} \ft v_{n_1}\cj{\ft v_{n_2}}\ft v_{n_3}
+ i |\ft v_n|^2 \ft v_n,
\label{4NLS1}
\end{align}

\noi
where the phase function $\phi(\bar n) $ is given by 
\begin{align}
\phi(\bar n) = \phi(n_1, n_2, n_3, n) = n_1^4 - n_2^4 + n_3^4 - n^4.
\label{phiX}
\end{align}

\noi
Note that this reduction of \eqref{4NLS0} to \eqref{4NLS1}
via two transformations on the phase space is  another instance of global analysis.
See Subsection \ref{SUBSEC:3.1}.
This reformulation exhibits resonant and non-resonant structure of the nonlinearity 
in an explicit manner
and moreover it removes certain resonant interactions, 
which was crucial in establishing an effective energy estimate in Step (i).
By applying a normal form reduction,
we introduced a {\it modified} energy $E_t = \|u(t)\|_{H^s}^2 + R_t$
for some appropriate correction term $R_t$. 
See \eqref{P1} - \eqref{P3} below.
We then established 
an energy estimate on the modified energy $E_t$, provided $s > \frac 34$.
In Step (ii), 
in order to justify such a change-of-variable formula, 
we considered a truncated dynamics.
Moreover, we needed to 
introduce and consider 
a change-of-variable formula for a {\it modified} measure associated with the modified energy
$E_t$ introduced in Step (i).

The regularity restriction $s > \frac 34$ in the previous paper \cite{OTz1}
comes from the energy estimate in Step (i), where we applied the normal form reduction
(namely integration by parts in time) once to the equation:
$\dt \| u \|_{H^s}^2 = \cdots$
satisfied by the $H^s$-energy functional $\| u \|_{H^s}^2$.
In the following, we  prove Theorem \ref{THM:1}
by performing normal form reductions {\it infinitely many times}.
Our  normal form approach is 
analogous to the approach employed in 
\cite{BIT, KO, GKO}.
In particular, in \cite{GKO}, the first author (with Guo and Kwon) implemented an infinite iteration 
scheme of normal form reductions to 
prove unconditional well-posedness 
of  the cubic NLS on $\T$ in low regularity.
In \cite{GKO},  we performed integration by parts in a successive manner, 
introducing nonlinear terms of higher and higher degrees.
While the nonlinear terms thus introduced are of higher degrees, 
they satisfy better estimates.
In order to keep track of all possible ways to perform integration by parts, 
we introduced the notion of ordered trees.
See also \cite{CGKO} for another example of an infinite iteration
of normal form reductions to prove unconditional well-posedness.

In establishing an improved energy estimate (Proposition \ref{PROP:energy}), 
we perform an infinite iteration of normal form reductions.
It is worthwhile to note that,
unlike \cite{GKO},  we do not work at the level of the equation \eqref{4NLS0}.
Instead, we work 
at the level of the evolution equation 
$\dt \| u \|_{H^s}^2 = \cdots$ satisfied by the $H^s$-energy functional.
Let us first go over the computation performed in \cite{OTz1}
to show a flavor of this method.
Using \eqref{4NLS1}, we have
\begin{align}
 \frac{d}{dt} \bigg(\frac 12\| u (t) \|_{H^s}^2\bigg)
& =  \frac{d}{dt} \bigg(\frac 12\|v(t) \|_{H^s}^2\bigg) \notag\\
& = - \Re i 
\sum_{n \in \Z} 
\sum_{\substack{n = n_1 - n_2 + n_3\\n \ne n_1, n_3}} 
e^{ - i \phi(\bar n) t} \jb{n}^{2s} \ft v_{n_1} \cj{\ft v_{n_2}} \ft v_{n_3} \cj{\ft v_n},
\label{P1}
\end{align}

\noi
where $v$ is the renormalized variable as in \eqref{4NLS1}.
Then, differentiating by parts, 
i.e.~integrating by parts without an integral symbol,\footnote{This is indeed  
a Poincar\'e-Dulac normal form reduction 
 applied to the evolution equation \eqref{P1} for $\frac 12 \|v(t)\|_{H^s}^2$.
See  Section 1 in \cite{GKO}
for a discussion on  the relation between differentiation by parts and normal form reductions.}
we obtain
\begin{align}
\frac{d}{dt} \bigg(\frac 12\|v(t) \|_{H^s}^2\bigg)
& = \Re \frac{d}{dt} \bigg[
\sum_{n \in \Z}\sum_{\substack{n = n_1 - n_2 + n_3\\n \ne n_1, n_3}} 
\frac{e^{ - i \phi(\bar n) t}}{ \phi(\bar n)} \jb{n}^{2s} \ft v_{n_1} \cj{\ft v_{n_2}} \ft v_{n_3} \cj{\ft v_n}\bigg]\notag \\
& \hphantom{X}
- \Re 
\sum_{n \in \Z}\sum_{\substack{n = n_1 - n_2 + n_3\\n \ne n_1, n_3}} 
\frac{e^{ - i \phi(\bar n) t}}{ \phi(\bar n)} \jb{n}^{2s} 
\dt \big(\ft v_{n_1} \cj{\ft v_{n_2}} \ft v_{n_3} \cj{\ft v_n}\big).
\label{P2}
\end{align}

\noi
This motivates us to define the first modified  energy $E_t^{(1)}(v)$ 
with the correction term $R_t^{(1)}(v)$ by
\begin{align}
E^{(1)}_t(v)  
 = & \, \frac 12 \|v\|_{H^s}^2  + R_t^{(1)}(v)\notag\\
  : = & \, \frac 12\|v\|_{H^s}^2 
- \Re 
\sum_{n \in \Z} \sum_{\substack{n = n_1 - n_2 + n_3\\n \ne n_1, n_3}}
\frac{e^{ - i \phi(\bar n) t}}{ \phi(\bar n)} \jb{n}^{2s} \ft v_{n_1} \cj{\ft v_{n_2}} \ft v_{n_3} \cj{\ft v_n} .
\label{P3}
\end{align}

\noi
This is  the modified energy used in the previous work \cite{OTz1}
(up to a constant factor).
Note that the time derivative of $E_t^{(1)}(v)$ is given by 
the second term
on the right-hand side of~\eqref{P2}.

In the second step, we divide the 
the second term
on the right-hand side of \eqref{P2}
into nearly resonant and non-resonant parts
and apply differentiation by parts only to the non-resonant part.
When we apply differentiation by parts as in \eqref{P2} in an iterative manner, 
 the time derivative may fall on any of the factors $\ft v_{n_j}$ and $\ft v_n$,
 generating higher order nonlinear terms.
In general, the structure of such terms can be very complicated, depending on where the time derivative falls.
In \cite{GKO}, ordered (ternary) trees played an important role for indexing such terms.
In our case, we work on the evolution equation satisfied
by the $H^s$-energy functional 
and we need to consider tree-like structures that grow in two directions.
In Section \ref{SEC:NF},  we introduce the notion of  bi-trees and ordered bi-trees
for this purpose.

After $J$ steps of the normal form reductions, we arrive at
\begin{align} 
\frac {d}{dt} \bigg(\frac 12 \| v (t) \|_{H^s}^2\bigg)
=   \frac {d}{dt}\bigg(   \sum_{j = 2}^{J+1}  \NN^{(j)}_0(v) & (t)\bigg)
  +\sum_{j = 2}^{J+1} \NN^{(j)}_1(v)(t) \notag\\
&  +  \sum_{j = 2}^{J+1} \RR^{(j)}(v)(t)
 + \NN^{(J+1)}_2(v)(t).
\label{P4a}
\end{align}

\noi
Here, $ \NN^{(j)}_0(v)$ consists of  certain $2j$-linear terms,
while 
 $ \NN^{(j)}_1(v)$ and $\RR^{(j)}(v)$  consist of $(2j+2)$-linear terms.
In practice, we obtain \eqref{P4a} for smooth functions
with a truncation parameter $N \in \N$.
Here,  we can only show that the remainder term $\NN^{(J+1)}_2$ satisfies the bound of the form:
\[ \big|\NN^{(J+1)}_2\big|\leq F(N, J),\]

\noi
with  the upper bound $F(N, J)$  satisfying
\[\lim_{N \to \infty} F(N, J) = \infty\] 

\noi
for each fixed $J \in \N$.
This, however, does not cause an issue since we also show that 
\[\lim_{J \to \infty} F(N, J) = 0\] 

\noi
for each {\it fixed} $N \in \N$.
Therefore, by first taking  the limit  $J \to \infty$  and then $N \to \infty$, 
we conclude that the error term $\NN^{(J+1)}_2$ vanishes in the limit.
See Subsection \ref{SUBSEC:error}.
While it is simple, this observation is crucial in an infinite iteration of the normal form reductions.

At the end of an infinite iteration of the normal form reductions,
we can rewrite \eqref{P1} as 
\begin{align} 
\frac {d}{dt} \bigg(\frac 12 \| v (t) \|_{H^s}^2\bigg)
=  \frac {d}{dt}\bigg( \sum_{j = 2}^\infty \NN^{(j)}_0(v)(t)\bigg)
+ \sum_{j = 2}^\infty \NN^{(j)}_1(v)(t) + \sum_{j = 2}^\infty \RR^{(j)}(v)(t)  , 
\label{P4}
\end{align}

\noi
involving infinite series.
The main point of this normal form approach 
is that, while 
the degrees of the nonlinear terms appearing in \eqref{P4} can be arbitrarily large,
we can show that they are all bounded in $L^2(\T)$ (in a summable manner over $j$).
In particular, by 
defining the modified energy $\EE_t(v)$
by 
\begin{align}
\EE_t(v) := \frac 12 \| v(t) \|_{H^s}^2- \sum_{j = 2}^\infty \NN^{(j)}_0(v)(t), 
\label{P5}
\end{align}

\noi
we see that its time derivative is bounded:
\begin{align*}
\bigg|\frac{d}{dt}\EE_t(v)\bigg| \leq C_s(\|u\|_{L^2}), 
\end{align*}

\noi
satisfying the energy estimate \eqref{P0} in Step (i) with $\theta = 2$.
See Proposition \ref{PROP:energy} below.
This is the main new ingredient for  proving Theorem \ref{THM:1}.
See also the recent work \cite{OW} by the first author (with Y.\,Wang)
on an infinite iteration of normal form reductions
for establishing a crucial energy estimate 
on the difference of two solutions in proving enhanced uniqueness 
for the renormalized cubic 4NLS (see \eqref{4NLS0a} below) in negative Sobolev spaces.

\begin{remark}\rm 
(i) Heuristically speaking, 
this infinite iteration of normal form reductions
allows us to  exchange analytical difficulty
with algebraic/combinatorial  difficulty.

\smallskip

\noi
(ii) The ``correction term'' $R_t^{(1)}$ in \eqref{P3}
is nothing but the correction term in the spirit of the $I$-method \cite{CKSTT1, CKSTT2}.
In fact, at each step  of  normal form reductions,  we obtain a correction term $\NN^{(j)}_0(v)$.
Hence, our improved energy estimate (Proposition \ref{PROP:energy})
via an infinite iteration of normal form reductions can be basically viewed
as an implementation of the $I$-method with 
an  infinite sequence $\big\{\NN^{(j)}_0(v)\big\}_{j = 2}^\infty$
of  correction terms. 
Namely, the modified energy $\EE_t(v)$ defined in \eqref{P5}
is a modified energy of an infinite order in the $I$-method terminology.\footnote{The highest order of modified energies
used in the literature  is three in the application of the $I$-method to the KdV equation \cite{CKSTT2}, corresponding 
to two iterations of normal form reductions.}

\smallskip

\noi
(iii) We point out that a  finite iteration of normal form reductions is not sufficient
to go below $s > \frac 34$.
See (6.14) in \cite{OTz1}, showing the restriction $s - \frac 12 > \frac 14$.

\end{remark}

\begin{remark}\rm 
Let us briefly discuss the situation for the more general cubic fourth order NLS \eqref{4NLS}.
For this equation, 
the following phase function
\begin{align*}
 	\phi_{\ld, \mu}(\bar n) 
	=  -\ld(n_1^2 - n_2^2 + n_3^2 - n^2)
	+ \mu(n_1^4 - n_2^4 + n_3^4 - n^4)
\end{align*}

\noi
plays an important role in the analysis.
In view of Lemma \ref{LEM:phase} below, 
we have
\begin{align}
\phi_{\ld, \mu}(\bar n) =  (n_1 - n_2)(n_1-n) 
\big\{- 2\ld + \mu \big(n_1^2 +n_2^2 +n_3^2 +n^2 + 2(n_1 +n_3)^2\big)\big\}.
 \label{phase1}
\end{align}

If the last factor in \eqref{phase1}
does not vanish for any $n_1, n_2, n_3, n \in \Z$, 
then we can establish quasi-invariance of $\mu_s$
under \eqref{4NLS} for $s > \frac 12$
with the same proof as in \cite{OTz1} and this paper.
It suffices to note that,  
while we make use of the divisor counting argument in the proof, 
we only apply it to 
$\mu(\bar n) =  (n_1 - n_2)(n_1-n) $
and thus the integer/non-integer character of the last factor in \eqref{phase1} is irrelevant.

For example, 
when $\ld \mu < 0$, the last factor in \eqref{phase1} does not vanish
and thus Theorem \ref{THM:1} applies to this case.
When $\ld \mu > 0$, 
the non-resonant condition 
$ 2\ld \not\in \mu \N$ 
also guarantees the non-vanishing of the last factor in \eqref{phase1}.
It seems of interest to investigate the transport property of the Gaussian measure $\mu_s$
in the resonant case $ 2\ld \in \mu \N$.
In this case, there are more resonant terms
and thus further analysis is required.

\end{remark}

\begin{remark}\rm

On the one hand, the cubic 4NLS \eqref{4NLS0} is ill-posed in negative Sobolev spaces
and hence the quasi-invariance result stated in Theorem \ref{THM:1} is sharp.
On the other hand, the first author and Y.\,Wang \cite{OW} considered the following renormalized
cubic 4NLS on $\T$:
\begin{align}
\textstyle i \dt u = \dx^4 u  + \big( |u|^2 -2  \fint |u|^2 dx\big) u,
\label{4NLS0a}
\end{align}

\noi
 where $\fint_\T f(x) dx := \frac{1}{2\pi} \int_\T f(x)dx$.
In particular, they proved global well-posedness
of \eqref{4NLS0a} in $H^s(\T)$ for $s > -\frac{1}{3}$.
In a very recent work \cite{OTW}, the first and third authors with Y.\,Wang
went further and constructed global-in-time dynamics
for \eqref{4NLS0a} 
almost surely with respect to   the white noise, i.e.~the Gaussian measure $\mu_s$ with $s = 0$
supported on $H^\s(\T)$, $\s < -\frac 12$.
As a result, they proved invariance of the white noise $\mu_0$
under the renormalized cubic 4NLS~\eqref{4NLS0a}.
 Invariance is of course a stronger property than quasi-invariance
 and hence the white noise is in particular quasi-invariant
 under~\eqref{4NLS0a}.
 The question of quasi-invariance of $\mu_s$ for $s \in (0, \frac 12]$ under the dynamics
 of the renormalized cubic 4NLS \eqref{4NLS0a} is therefore 
 a natural sequel of the analysis of this paper.

 \end{remark}

\subsection{Non quasi-invariance  under the dispersionless model}

To motivate our second result, note that, by invariance of the complex-valued Gaussian random variable
$g_n$ in~\eqref{gauss1} under rotations, it is clear that 
the Gaussian measure $\mu_s$ is  invariant under the linear dynamics:
\begin{equation}
\begin{cases}
i \partial_t u = \partial_x^4 u\\
u|_{t=0}= u_0,
\end{cases}  \qquad (x,t)\in \mathbb{T}\times \mathbb{R}.
\label{linear}
\end{equation}

\noi
See Lemma \ref{LEM:gauss} (i) below.
In particular,  $\mu_s$ is quasi-invariant under the linear dynamics~\eqref{linear}. 

In the proof of the quasi-invariance of $\mu_s$ under the cubic 4NLS \eqref{4NLS0} (Theorem \ref{THM:1} above), 
the dispersion plays an essential role.
The strong dispersion allows us to show that the nonlinear part in \eqref{4NLS0}
is a perturbation to the linear equation \eqref{linear}.
Our next result shows that the dispersion is indeed {\it essential} for Theorem \ref{THM:1} to hold.

Consider the following dispersionless model:
\begin{equation}\label{ND1}
\begin{cases}
i \partial_t u = |u|^2u\\
u|_{t=0}= u_0,
\end{cases}
  \qquad (x,t)\in \mathbb{T}\times \mathbb{R}.
\end{equation}

\noi
Recall that there is an explicit solution formula for \eqref{ND1} given by:
\begin{equation}\label{formula}
u(x,t)=e^{- it|u_0(x)|^2}u_0(x)
\end{equation}

\noi
at least for continuous initial data
such that the pointwise product makes sense.

Let $s > \frac 12$.
Then, it is easy to see that the random function $u^\o$ in \eqref{gauss1} is continuous almost surely.  
Indeed, by the equivalence of Gaussian moments
and the mean value theorem, we have
\begin{align*}
\E\big[|u^\o(x) - u^\o(y)|^p\big] 
& \leq C_p 
\Big(\E\big[|u^\o(x) - u^\o(y)|^2\big]\Big)^\frac{p}{2}
\sim \bigg(\sum_{n \in \Z}\frac{1}{\jb{n}^{2s - 2\eps}}\bigg)^\frac{p}{2}
|x - y|^{\eps p}\\
& \les |x - y|^{\eps p }, 
\end{align*}

\noi
provided that $\eps > 0$ is sufficiently small such that $2s - 2\eps > 1$.
Now, by choosing $p \gg1 $ such that $\eps  p > 1$, 
we can apply Kolmogorov's continuity criterion
and conclude that $u^\o$ in~\eqref{gauss1} is almost surely continuous when $s > \frac 12$.
This in particular implies that 
 the solution formula  \eqref{formula} is well defined for initial data distributed according to $\mu_s$, $s > \frac 12$, 
 and the corresponding solutions exist globally in time.
We denote by  $\wt \Phi(t)$ 
the solution map for the dispersionless model \eqref{ND1}.

We now state our second result.

\begin{theorem}\label{THM:sing}
Let $s > \frac 12$.
Then, 
given $t \ne 0$, 
the pushforward measure $\wt \Phi(t)_* \mu_s$
under the dynamics of the dispersionless model \eqref{ND1}
is not absolutely continuous with respect to the Gaussian measure $\mu_s$. 
Namely, the Gaussian measure $\mu_s$ is not quasi-invariant
under the dispersionless dynamics \eqref{ND1}.
\end{theorem}

This is a sharp contrast with the quasi-invariance result
for the cubic 4NLS in Theorem~\ref{THM:1}
and for the cubic NLS
for $s \in \N$ (see Remark 1.4 in \cite{OTz1}).
In particular, Theorem~\ref{THM:sing} shows that dispersion is essential
for establishing quasi-invariance of $\mu_s$.

We prove this negative result in Theorem \ref{THM:sing} 
by establishing 
that typical elements under $\mu_s$, $s>\frac 12 $, possess an almost surely constant modulus of continuity at each point. This is the analogue of the classical law of the iterated logarithm for the Brownian motion. 
We show that  this modulus of continuity is 
destroyed with a positive probability by the nonlinear transformation \eqref{formula}
for any non-zero time $t \in \R \setminus \{0\}$.

Our proof is based on  three basic tools:
the Fourier series representation of the (fractional)
Brownian loops,  the  law of the iterated logarithm,
and the solution formula \eqref{formula} to the dispersionless model \eqref{ND1}.
We will use three different versions of the law of the iterated logarithm,
depending on 
(i) $s = 1$ corresponding to the Brownian/Ornstein-Uhlenbeck loop,
(ii) $\frac 12 < s < \frac 32$, corresponding to the fractional Brownian loop, 
and (iii) the critical case $s = \frac 32$.
In Cases (i) and (ii), we make use of the mutual absolute continuity
of the function $u$ given in the random Fourier series \eqref{gauss1}
and the (fractional) Brownian motion on $[0, 2\pi)$
to  deduce  the  law of the iterated logarithm for the random function $u$.
In Case~(iii), we directly establish the relevant 
law of the iterated logarithm for $u$ in \eqref{gauss1}.\footnote{We could apply 
this argument to directly establish the relevant law of the iterated logarithm 
in Cases~(i) and (ii) as well.}  See Proposition~\ref{PROP:LL}.

On the one hand, 
the law of the iterated logarithm
yields 
almost sure constancy of the  modulus of continuity  at time $t = 0$.
On the other hand, 
we combine this almost sure constancy of the  modulus of continuity at time $t = 0$
and the solution formula \eqref{formula}
to show that the modulus of continuity at non-zero time $t\ne 0$
does not satisfy the conclusion of the law of the iterated logarithm
with a positive probability.
Lastly, for $s > \frac 32$, we reduce the proof to  one of Cases (i), (ii), or (iii)
by differentiating the random function.

\begin{remark}\rm

The existence of a quasi-invariant measure
shows a delicate persistence
property of the dynamics.
In particular,  this persistence property due to the quasi-invariance
is stronger than the  persistence of regularity\footnote{
In the scaling sub-critical case, 
by persistence of regularity, we  mean the following;
if one proves local well-posedness in $H^{s_0}$ for some $s_0 \in \R$
and if $u_0$ lies in a smoother space $H^s$ for some $s > s_0$, 
then the corresponding solution remains smoother
and lies  in $C([-T, T] ;H^s)$,
where 
the local existence time $T>0$ depends only on the $H^{s_0}$-norm of the initial condition $u_0$.
}
obtained by the usual well-posedness theory.
While the dispersionless model \eqref{ND1}
enjoys the persistence of regularity in $H^\s(\T)$, $\s > \frac 12$, 
Theorem \ref{THM:sing} shows that the Gaussian measure $\mu_s$ 
is not quasi-invariant 
under the dynamics of \eqref{ND1}.

\end{remark}

\begin{remark}\label{REM:eps}\rm
(i) 
For $\eps \in \R$, 
consider the  following 4NLS:
\begin{equation}
i \dt u = \eps \dx^4 u + |u|^2u.
\label{eNLS}
\end{equation}

\noi
For smooth initial data, 
it is easy to show that,
on the unit time interval $[0, 1]$,  
the corresponding solutions to the small dispersion 4NLS,
i.e.~\eqref{eNLS} with small $\eps \ne0$, 
converge to those to the dispersionless model~\eqref{ND1}
as $\eps \to 0$.
See Lemma 4.1 in \cite{OW0}.
In this sense, 
the small dispersion 4NLS~\eqref{eNLS} with $|\eps|\ll 1$
is ``close'' to the dispersionless model~\eqref{ND1}.

On the other hand, there is a dichotomy 
in the statistical behavior of solutions
to the small dispersion 4NLS \eqref{eNLS} 
and  to the dispersionless model \eqref{ND1}.
When $\eps \ne 0$, one can easily adapt the proof of Theorem \ref{THM:1}
and prove quasi-invariance of  the Gaussian measure $\mu_s$.
When $\eps = 0$, however, Theorem  \ref{THM:sing} shows that $\mu_s$ is not quasi-invariant under
\eqref{ND1}.
This shows 
 a dichotomy between quasi-invariance for $\eps \ne 0$
and non quasi-invariance for $\eps = 0$, 
while there is a good approximation property
for the deterministic dynamics of the dispersionless model \eqref{ND1}
by that of the small dispersion 4NLS \eqref{eNLS} with $|\eps|\ll1$.

\smallskip

\noi
(ii) We mention recent work 
\cite{OTsTz, FT}
on establishing quasi-invariance of the Gaussian measure $\mu_s$
for Schr\"odinger-type equations with less dispersion.
In particular, Forlano-Trenberth \cite{FT} studied
the following fractional NLS:
\begin{equation}
i \dt u = (-\dx^2)^\frac{\al}{2} u + |u|^2u
\label{fNLS}
\end{equation}

\noi
and proved quasi-invariance of $\mu_s$ (for some non-optimal range of $s >s_\al$),
provided that $\al > 1$.
When $\al = 1$, 
the equation \eqref{fNLS} corresponds to the half-wave equation, 
which  does not possess any dispersion.
See \cite{GG, Poc, GLPR}.
It would be of interest to study the transport properties
of the Gaussian measure $\mu_s$ under \eqref{fNLS}
for $0 < \al < 1$.
We also point out  that 
 \eqref{fNLS}
for $0 < \al < 1$ also appears as a model in the study of one-dimensional wave turbulence
\cite{MMT}
and hence is a natural model for statistical study of its solutions.

\end{remark}

\subsection{Organization of the paper}

In Section~\ref{SEC:notations}, we introduce some notations.
In Section~\ref{SEC:3}, we prove Theorem \ref{THM:1},
assuming the improved energy estimate (Proposition \ref{PROP:energy}).
We then present the proof of  Proposition \ref{PROP:energy}
in Section \ref{SEC:NF}
by implementing an infinite iteration of normal form reductions.
Lastly, by studying the random Fourier series \eqref{gauss1} and the relevant law
of the iterated logarithm, 
we prove 
non quasi-invariance of 
the Gaussian measure $\mu_s$ 
under the dispersionless model (Theorem \ref{THM:sing}) in Section \ref{SEC:sing}.

\section{Notations}
\label{SEC:notations}

Given $N \in \N$, 
we use $\P_{\leq N}$ to denote the Dirichlet projection 
onto the frequencies $\{|n|\leq N\}$
and  set $\P_{> N} := \text{Id} - \P_{\leq N}$.
When $N = \infty$, it is understood that $P_{\le N} = \text{Id}$.
Define $E_N$ and $E_N^\perp$ by 
\begin{align*}
E_N & = \P_{\leq N} L^2(\T) = \text{span}\{e^{inx}: |n|\leq N\},\\
E_N^\perp & = \P_{>N} L^2(\T) = 
\text{span}\{e^{inx}: |n|> N\}.
\end{align*}

Given $ s> \frac{1}{2}$, let $\mu_s$ be the  Gaussian measure on $L^2(\T)$
defined in \eqref{gauss0}.
Then, we can write $\mu_s$ as
\begin{align*}
 \mu_s = \mu_{s, N}\otimes \mu_{s, N}^\perp,
 \end{align*}

\noi
where $ \mu_{s, N}$ and $\mu_{s, N}^\perp$ 
are the 
marginal distributions of $\mu_s$ restricted onto $E_N$ and $E_N^\perp$, respectively.
In other words, 
$ \mu_{s, N}$ and $\mu_{s, N}^\perp$ 
are 
the induced probability measures
under the following maps: 
\begin{align*}
& u_N :\o \in \O \longmapsto  u_N (x; \o) = \sum_{|n|\leq N} \frac{g_n(\o)}{\jb{n}^s}e^{inx}, \\
& u_N^\perp: \o \in \O \longmapsto u_N^\perp(x; \o) = \sum_{|n|>N} \frac{g_n(\o)}{\jb{n}^s}e^{inx}, 
 \end{align*}

\noi
 respectively.
Formally, we can write  $ \mu_{s, N}$ and $\mu_{s, N}^\perp$
as
\begin{align}
 d \mu_{s, N} = Z_{s, N}^{-1} e^{-\frac 12 \| \P_{\leq N}  u_N\|_{H^s}^2} d u_N 
\quad \text{and} \quad  
d \mu_{s, N}^\perp = \ft Z_{s,N }^{-1} e^{-\frac 12 \| \P_{>N} u_N^\perp \|_{H^s}^2} d u_N^\perp. 
\label{G4}
 \end{align}

\noi
Given $r > 0$, we also define a probability measure $\mu_{s, r}$ 
with an $L^2$-cutoff
by 
\begin{align*}
d \mu_{s, r} = Z_{s, r}^{-1}\ind_{\{ \| v\|_{L^2 } \leq r\}} d\mu_s.
\end{align*}

Given a function $v \in L^2(\T)$, 
we simply use  $v_n$ to denote the Fourier coefficient $\ft v_n$ of $v$,
when there is no confusion.
This shorthand notation is especially useful in Section \ref{SEC:NF}.

We  use $a+$ (and $a-$) to denote $a + \eps$ (and $a - \eps$, respectively) for arbitrarily small $\eps \ll 1$,
where an implicit constant is allowed to depend on $\eps > 0$
(and it usually diverges as $\eps \to 0$).
Given $x \in \R$, 
we use 
   $\lfloor x \rfloor $ to denote the integer part
   of $x$.

In view of the time reversibility of the equations \eqref{4NLS0} and \eqref{ND1}, 
we only consider positive times in the following.

\section{Proof of Theorem \ref{THM:1}:
Quasi-invariance of $\mu_s$ under the cubic 4NLS}
\label{SEC:3}

In this section, we present  the proof of Theorem \ref{THM:1}.
The main new ingredient is the improved energy estimate (Proposition \ref{PROP:energy})
whose proof is postponed to Section \ref{SEC:NF}.
The remaining part of the proof follows closely
 the presentation in \cite{OTz1}
and thus we keep our discussion concise.

\subsection{Basic reduction of the problem}
\label{SUBSEC:3.1}

We first go over the basic reduction of the problem from \cite{OTz1}.
Given $t \in \R$, we define a gauge transformation $\GG_t$ on $L^2(\T)$
by setting
\begin{align*}
 \mathcal{G}_t [f ]: = e^{ 2 i t \fint |f|^2} f.
\end{align*}

\noi
Given a function $u \in C(\R; L^2(\T))$, 
we define $\GG$ by setting
\[\GG[u](t) : = \GG_t[u(t)].\]

\noi
Note that $\GG$ is invertible
and its inverse is given by $\GG^{-1}[u](t) = \GG_{-t}[u(t)]$.

Given a solution  $u \in C(\R; L^2(\T))$  to \eqref{4NLS0}, 
let  $\wt u = \GG[u]$.
Then, it follows from  the mass conservation 
that $\wt u$ is a solution to the following renormalized fourth order NLS:
\begin{align}
\textstyle i \dt \wt u  =   \dx^4 \wt u  +\big( |\wt u |^{2}  - 2 \fint_\T |\wt u |^2dx \big) \wt u. 
\label{NLS4}
\end{align}

\noi
This is precisely the renormalized cubic 4NLS in \eqref{4NLS0a}.

Let $S(t) = e^{-it \dx^4}$ be the solution operator for the linear fourth order Schr\"odinger equation~\eqref{linear}.
Denoting by $v= S(-t) \wt u$ the interaction representation of  $\wt u$, 
we see that $v$ satisfies 
the following equation for $\{v_n\}_{n \in \Z}$: 
\begin{align}
\dt v_n 
& = -i \sum_{\G(n)} e^{-i \phi(\bar n) t} v_{n_1}\cj{v_{n_2}}v_{n_3}
+ i |v_n|^2 v_n \notag\\
& =: \NN(v)_n + \RR(v)_n, 
\label{NLS5}
\end{align}

\noi
where the phase function $\phi(\bar n)$ is as in \eqref{phiX}
and the plane $\G(n)$ is  given by
\begin{align}
\G(n) 
= \{(n_1, n_2, n_3) \in \Z^3:\, 
 n = n_1 - n_2 + n_3 \text{ and }  n_1, n_3 \ne n\}.
 \label{Gam1}
 \end{align}

Recall that the phase function $\phi(\bar n)$ admits
the following factorization.
See \cite{OTz1} for the proof.
	
\begin{lemma}\label{LEM:phase}
Let $n = n_1 - n_2 + n_3$.
Then, we have
\begin{align*}
\phi(\bar n) =  (n_1 - n_2)(n_1-n) 
\big( n_1^2 +n_2^2 +n_3^2 +n^2 + 2(n_1 +n_3)^2\big).
\end{align*}
	
\end{lemma}

It follows from Lemma \ref{LEM:phase} that 
$\phi(\bar n) \ne 0$ on $\G(\bar n)$.
Namely, $\NN(v)$ and $\RR(v)$
on the right-hand side of \eqref{NLS5}
correspond to the non-resonant and resonant parts, respectively.
It follows from   Lemma \ref{LEM:phase}  that 
there is a strong smoothing property on the non-resonant  term $\NN(v)$
due to the fast oscillation caused by $\phi(\bar n)$.

\medskip

Given $t, \tau \in \R$, let 
$\Phi(t):L^2(\T)\to L^2(\T)$ be the solution map for \eqref{4NLS0}
and 
$ \Psi(t, \tau):L^2(\T)\to L^2(\T)$ be the solution map for \eqref{NLS5},\footnote{Note that \eqref{NLS5}
is non-autonomous.
We point out that this non-autonomy does not play an essential
role in the remaining part of the paper, since all the estimates hold uniformly in $t \in \R$.
}
sending initial data at time $\tau$
to solutions at time $t$. 
When $\tau =0$, we may denote $\Psi(t, 0)$ by $\Psi(t)$ for simplicity.
Then, from $v = S(-t)\circ \GG[u]$, we have 
\begin{align}
\Phi(t) = \GG^{-1} \circ S(t) \circ \Psi(t).
\label{nonlin2}
\end{align}

\noi
Recall the following lemma from \cite{OTz1}.

\begin{lemma}\label{LEM:gauss}
\textup{(i)}
Let $s>\frac 12$
and $t \in \R$.  Then, the Gaussian measure $\mu_s$ defined in \eqref{gauss0} is invariant under 
the linear map $S(t)$
and the gauge transformation $\GG_t$.

\smallskip

\noi
\textup{(ii)}
Let $(X, \mu)$ be a measure space.
Suppose that $T_1$ and $T_2$
are measurable maps on $X$ into itself
such that $\mu$ is quasi-invariant under $T_j$ for each $j = 1, 2$.
Then, $\mu$ is quasi-invariant under  $T = T_1 \circ T_2$.

\end{lemma}

When $s = 1$, Lemma \ref{LEM:gauss} (i) basically follows 
from Theorem 3.1 in \cite{NRBSS} which exploits the properties of the Brownian loop
under conformal mappings.
For general  $s > \frac 12$, 
such approach does not seem to be appropriate.
See Section 4 in \cite{OTz1} for the proof of Lemma~\ref{LEM:gauss}.

In view of this lemma, 
Theorem \ref{THM:1} follows once we  prove quasi-invariance of $\mu_s$ under $\Psi(t)$.
Therefore,  we focus our attention to  \eqref{NLS5}
in the following.

\subsection{Truncated dynamics}

Let us first introduce the following truncated approximation to \eqref{NLS5}:
\begin{align}
\dt v_n 
& = 
\ind_{|n|\leq N}\bigg\{-i 
\sum_{\G_N(n)} e^{-i \phi(\bar n) t} v_{n_1}\cj{v_{n_2}}v_{n_3}
+ i |v_n|^2 v_n\bigg\},  
\label{NLS6}
\end{align}

\noi
where $\G_N(n)$ is defined by 
\begin{align*}
\G_N(n) & = \G(n) \cap \{ (n_1, n_2, n_3) \in \Z^3: |n_j|\leq N\}\notag\\
& = \{(n_1, n_2, n_3) \in \Z^3:\, 
 n = n_1 - n_2 + n_3,  \,  n_1, n_3, \ne n, 
 \text{ and } |n_j| \leq N\}.
\end{align*}

\noi
Note that \eqref{NLS6} is an infinite dimensional system of ODEs
for the Fourier coefficients $\{ v_n \}_{n \in \Z}$, 
where the flow is constant on the high frequencies $\{|n|> N\}$.
We also consider the following finite dimensional system of ODEs:
\begin{align}
\dt v_n = 
-i 
\sum_{\G_N(n)} e^{-i \phi(\bar n) t} v_{n_1}\cj{v_{n_2}}v_{n_3}
+ i |v_n|^2 v_n,  \qquad |n| \leq N.
\label{NLS7}
\end{align}

\noi
with $v|_{t = 0} = \P_{\leq N}v|_{t = 0}$,
i.e.~$v_n|_{t=0} = 0$ for $|n| > N$.

Given $t, \tau \in \R$, 
 denote by $\Psi_N(t, \tau)$ and $\wt \Psi_N(t, \tau)$
the solution maps of \eqref{NLS6}
and \eqref{NLS7}, 
sending initial data at time $\tau$
to solutions at time $t$, respectively.
For simplicity, we set 
\begin{align}\Psi_N(t) = \Psi_N(t, 0)
\qquad \text{and}
\qquad 
\wt \Psi_N(t) = \wt \Psi_N(t, 0)
\label{Psit}
\end{align}

\noi
 when $\tau = 0$. 
Then, we have the following relations:
\begin{align*}
 \Psi_N(t, \tau) = \wt \Psi_N (t, \tau) \P_{\leq N} + \P_{>N}
\qquad \text{and}
\qquad
\P_{\leq N} \Psi_N(t, \tau) = \wt \Psi_N (t, \tau) \P_{\leq N}.
 \end{align*}

We now recall the following approximation property 
of the truncated dynamics \eqref{NLS6}.

\begin{lemma}[Proposition 6.21/B.3 in \cite{OTz1}]\label{LEM:approx}

Given $R>0$, let $A \subset B_R$ be a compact set in $L^2(\T)$.
Fix $t \in \R$.
Then, for any $\eps > 0$, there exists $N_0 = N_0(t, R, \eps) \in \N$ such that we have
\begin{align*} 
\Psi(t) (A)\subset \Psi_N(t) (A + B_\eps)
\end{align*}

\noi
for all $N \geq N_0$.
Here,  $B_r$  denotes 
the ball  in $ L^2(\T)$ of radius $r$ centered at the origin.

\end{lemma}

\subsection{Energy estimate}\label{SUBSEC:energy}

In this subsection, we state a crucial energy estimate.
The main goal is to establish an energy estimate of the form \eqref{P0}
by introducing a suitable modified energy functional.
We achieve this goal by performing normal form reductions
infinitely many times and
thus constructing an infinite sequence of correction terms.

Let  $N \in \N \cup\{\infty\}$.
In the following, 
we simply say that $v$ is a solution 
to \eqref{NLS7} 
if  $v$ is  a solution to \eqref{NLS7} when $N \in \N$
  and  to \eqref{NLS5} when $N = \infty$.

\begin{proposition}\label{PROP:energy}
Let $\frac 12< s < 1$. 
Then, given  $N \in \N \cup\{\infty\}$, 
 there exist multilinear forms 
$\big\{\NN_{0, N}^{(j)}\big\}_{j = 2}^\infty$, 
$\big\{\NN_{1, N}^{(j)}\big\}_{j = 2}^\infty$, 
and 
$\big\{\RR_N^{(j)}\big\}_{j = 2}^\infty$ 
such that 
\begin{align*} 
\frac {d}{dt} \bigg(\frac 12 \| v (t) \|_{H^s}^2\bigg)
=  \frac {d}{dt}\bigg( \sum_{j = 2}^\infty \NN^{(j)}_{0, N}(v)(t)\bigg)
 + \sum_{j = 2}^\infty \NN^{(j)}_{1, N}(v)(t) + \sum_{j = 2}^\infty \RR_N^{(j)}(v)(t), 
\end{align*}

\noi
for any solution  $v \in C(\R; H^s(\T))$ to \eqref{NLS7}.
Here, 
$\NN_{0, N}^{(j)}$ are $2j$-linear forms, 
while
 $\NN_{1, N}^{(j)}$ and   $\RR_N^{(j)}$  are $(2j+2)$-linear forms, satisfying the following bounds on $L^2(\T)$;
 there exist positive constants $C_0(j)$, $C_1(j)$, and $C_2(j)$
 decaying faster than any exponential  rate\footnote{In fact, 
by slightly modifying the proof,   we can make  $C_0(j)$, $C_1(j)$, and $C_2(j)$
 decay as fast as we want as $j \to \infty$.}
  as $j \to \infty$ 
 such that 
\begin{align*} 
| \NN^{(j)}_{0, N}(v)(t)|
& \les
C_0(j)  \|v\|_{L^2}^{2j}, 
 \\
| \NN^{(j)}_{1, N}(v)(t)| 
& \les
C_1(j) \|v\|_{L^2}^{2j+2},
\\
| \RR^{(j)}_N(v)(t)|
&  \les
C_2(j) \|v\|_{L^2}^{2j+2},
\end{align*}

\noi
for $j = 2, 3, \dots$.
Note that these constants
$C_0(j)$, $C_1(j)$, and $C_2(j)$
are  independent of the cutoff size $N \in \N \cup \{\infty\}$
and $t \in \R$.

Define the modified energy $\EE_{N, t}(v)$
by 
\begin{align}
\EE_{N, t}(v) := \frac 12 \| v(t) \|_{H^s}^2- \sum_{j = 2}^\infty \NN^{(j)}_{0, N}(v)(t).
\label{energy6}
\end{align}

\noi
Then, 
the following energy estimate holds:
\begin{align}
\bigg|\frac{d}{dt}\EE_{N, t}(v)\bigg| \leq C_s(\|v\|_{L^2})
\label{energy7}
\end{align}

\noi
for any solution  $v \in C(\R; H^s(\T))$ to \eqref{NLS7},
uniformly in $N \in \N \cup \{\infty\}$
and   $t \in \R$.

\end{proposition}

In the remaining part of this section,  we continue with the proof of Theorem \ref{THM:1},
assuming Proposition~\ref{PROP:energy}.
We present the proof of Proposition~\ref{PROP:energy}
in Section~\ref{SEC:NF}.
 See Lemmas~\ref{LEM:N^J+1_0} 
and~\ref{LEM:N^J+1_1}.
In the following, we simply denote $\EE_{\infty, t}$
by $\EE_{ t}$
and  drop the subscript $N = \infty$ from the multilinear forms,
when $N = \infty$.
For example, we write 
$ \NN^{(j)}_{0}$ for $\NN^{(j)}_{0, \infty}$.

\subsection{Weighted Gaussian measures}

Let $s > \frac 12$.
As in \cite{OTz1}, we would like to  define the weighted Gaussian measures
associated with the modified energies $\EE_{N, t}(v)$ and $\EE_t(v)$:\footnote{Noting that we have $\P_{\leq N} v = v$ for all solutions to \eqref{NLS7}, 
we have $\EE_{N, t}(\P_{\leq N} v) = \EE_{N, t}( v)$.
In the following, we explicitly insert $\P_{\leq N}$ for clarity.
A similar comment applies to $\NN^{(j)}_{0, N}$.}
\begin{align}
d \rho_{s, N, r, t }
&  = \text{``}Z_{s, N, r}^{-1} \ind_{\{ \| v\|_{L^2 } \leq r\}} e^{-  \EE_{N, t}(\P_{\leq N} v)} dv\text{''} \notag\\
& =  
Z_{s, N, r}^{-1} F_{N, r, t} d\mu_s
\label{Gibbs1}
\end{align}

\noi
and 
\begin{align}
d \rho_{s, r, t}
&  = \text{``}Z_{s, r}^{-1} \ind_{\{ \| v\|_{L^2 } \leq r\}} e^{-  \EE_t( v)} dv\text{''} \notag\\
& =  
  Z_{s, r}^{-1} F_{r, t} d\mu_s,
\label{Gibbs1a}
\end{align}

\noi
where $\EE_{N, t}$ is the modified energy defined in \eqref{energy6}
and 
$F_{N, r, t }$ and $F_{r, t }$  are given by 
\begin{align}
F_{N, r, t }(v) & := \ind_{\{ \| v\|_{L^2 } \leq r\}} \exp\bigg( \sum_{j = 2}^\infty \NN^{(j)}_{0, N}(\P_{\leq N} v)(t)\bigg),
\label{Gibbs2}\\
F_{ r, t }(v) & = 
F_{ r,\infty,  t }(v): =
\ind_{\{ \| v\|_{L^2 } \leq r\}} \exp\bigg( \sum_{j = 2}^\infty \NN^{(j)}_0( v)(t)\bigg).\notag 
\end{align}

It follows from Proposition \ref{PROP:energy}
that
\begin{align*}
\ind_{\{ \| v\|_{L^2 } \leq r\}}
\exp\bigg( \sum_{j = 2}^\infty \big|\NN^{(j)}_{0, N}( \P_{\leq N}v)\big|\bigg)
& \leq \exp\bigg(\sum_{j = 2}^\infty
C_0(j)  r^{2j}\bigg)
\leq C(s, r),
\end{align*}

\noi
uniformly in $N \in \mathbb{N} \cup\{\infty\}$ and $t \in \R$.	
Hence, we have
\begin{align*}
Z_{s, N, r} = \int_{H^{s-\frac{1}{2}-}} F_{N, r, t}(v) d\mu_s \leq C(s, r),
\end{align*}

\noi
uniformly in $N \in \mathbb{N}\cup\{\infty\}$ and $t \in \R$.	
See also Remark \ref{REM:t1} below.
This shows that 
$\rho_{s, N, r, t}$ and 
$\rho_{s, r, t}$ in \eqref{Gibbs1} and \eqref{Gibbs1a} are well defined probability measures
 on $H^{s-\frac 12 -\eps}(\T)$, $\eps > 0$.
Moreover, the following lemma immediately follows from the computation above
as in  \cite[Proposition 6.2 and Corollary 6.3]{OTz1}.
See Subsection \ref{SUBSEC:finite}.

\begin{lemma}\label{LEM:Gibbs}
Let $s > \frac 12$ and $r > 0$.

\smallskip

\noi
\textup{(i)}
Given  any finite $p \geq 1$, 
$F_{N, r, t}(v)$ converges to  $F_{r, t}(v)$ in $L^p (\mu_s)$,
uniformly in $t\in \R$,  as $N \to \infty$.

\smallskip

\noi
\textup{(ii)}
 For any  $\g > 0$, there exists $N_0 \in \N$ such that 
\[ |  \rho_{s, N, r, t}(A) - \rho_{s, r, t}(A)| < \g \]

\noi
for any $N \geq N_0$
and  any measurable set 
$A \subset L^2(\T)$,
uniformly in $t\in \R$.
\end{lemma}

\begin{remark}\label{REM:t1}\rm 
The normalizing constants
$Z_{s, N, r}$ and $Z_{s, r}$
a priori depend on $t \in \R$.
It is, however, easy to see that 
they are indeed independent of $t \in \R$
by (i) noticing that 
the correction terms $\big\{\NN^{(j)}_{0, N}\big\}_{j = 2}^\infty$ defined in 
Proposition \ref{PROP:energy}
is in fact autonomous in terms of $\wt u (t)= S(t) v(t)$
and 
(ii)  the invariance of $\mu_s$ under $S(t)$ (Lemma \ref{LEM:gauss}).
See also Remark \ref{REM:tdepend} below.
The same comment applies to the normalizing constant 
$\ft Z_{s, N, r}$ defined in \eqref{norm1}.

\end{remark}

\subsection{A change-of-variable formula}

Next, we go over an important global aspect of the proof of Theorem \ref{THM:1}.
Given $N \in \N$, let $d L_N = \prod_{|n|\leq N} d\ft u_n $ denote the Lebesgue
measure on $E_N \cong \C^{2N+1}$.
Then, 
from \eqref{Gibbs1}
and \eqref{Gibbs2} with \eqref{G4}, we have 
\begin{align*}
d \rho_{s, N, r, t} 
& = Z_{s, N, r}^{-1} 
\ind_{\{ \| v\|_{L^2 } \leq r\}} e^{ \sum_{j = 2}^\infty \NN^{(j)}_{0, N}(\P_{\leq N} v)}
 d\mu_s\notag\\
& =\ft Z_{s, N, r}^{-1} 
\ind_{\{ \| v\|_{L^2 } \leq r\}} e^{- \EE_{N, t}(\P_{\leq N} v)}
dL_N \otimes
 d\mu_{s, N}^\perp, 
\end{align*}

\noi
where $\ft Z_{s, N, r}$ is a normalizing constant defined by
\begin{align}
\ft Z_{s, N, r}  = \int_{L^2} \ind_{\{ \| v\|_{L^2 } \leq r\}} 
e^{- \EE_{N, t}( \P_{\leq N}  v)} d L_N \otimes d\mu_{s, N}^\perp.
\label{norm1}
\end{align}

\noi
Then, proceeding as in \cite{OTz1}
and exploiting invariance of  $L_N$ under the map $\wt \Psi_N(t, \tau)$ for~\eqref{NLS7}, 
we have the following change-of-variable formula.

\begin{lemma}\label{LEM:meas1}
Let $s > \frac 12$, $N \in \N$, and $r > 0$.
Then, we have 
\begin{align*}
\rho_{s, N, r, t}(\Psi_N(t, \tau)(A))
& = Z_{s, N, r}^{-1}\int_{\Psi_N(t, \tau)(A)} \ind_{\{ \| v\|_{L^2 } \leq r\}} e^{\sum_{j = 2}^\infty \NN^{(j)}_{0, N}( \P_{\leq N} v)} d\mu_s(v)
\notag \\
& = \ft Z_{s, N, r}^{-1}\int_{A} \ind_{\{ \| v\|_{L^2 } \leq r\}} 
e^{- \EE_{N, t}( \P_{\leq N} \Psi_N(t, \tau) (v))} d L_N \otimes d\mu_{s, N}^\perp
\end{align*}

\noi	
for any $t, \tau \in \R$ and any measurable set 
$A \subset L^2(\T)$.
Here, $\Psi_N(t, \tau)$ is the solution map to~\eqref{NLS6}
defined in 	\eqref{Psit}.
\end{lemma}

\subsection{On the measure evolution property
and the proof of Theorem \ref{THM:1}}\label{SUBSEC:evo}

In this subsection, we 
use the energy estimate (Proposition \ref{PROP:energy})
and the change-of-variable formula (Lemma  \ref{LEM:meas1})
to establish a growth estimate
on the truncated weighted Gaussian measure $\rho_{s, N, r, t}$
under  $\Psi_N(t) = \Psi_N(t, 0)$ for \eqref{NLS6}.
Thanks to the improved energy estimate, 
the following estimates are simpler than those presented in \cite{OTz1}.

\begin{lemma}\label{LEM:meas2}
Let $ \frac 12 < s < 1$. 
Then,  given  $r > 0$, 
there exists $C = C(r)>0$ such that 
\begin{align}
\frac{d}{dt} \rho_{s, N, r, t}(\Psi_N(t) (A))
\leq C  \rho_{s, N, r, t} (\Psi_N(t)(A))
\label{meas2}
\end{align}

\noi
for any $N \in \N$, any $t \in \R$, 
and  any measurable set 
$A \subset L^2(\T)$.
As a consequence, we have the following estimate;
given  $t \in \R$ and $r > 0$, 
 there exists $C = C(t,  r) > 0$
such that 
\begin{align}
\rho_{s, N, r, t} (\Psi_N(t) (A)) \leq C\rho_{s, N, r, t} ( A) 
\label{meas3}
\end{align}
	
\noi
for 
any $N \in \N$
and  any measurable set 
$A \subset L^2(\T)$.

\end{lemma}

\begin{proof}
 As in \cite{TzV1, TzV2, TzBBM, OTz1}, 
the main idea of the proof of Lemma \ref{LEM:meas2}
is to reduce the analysis to that at $t = 0$ in the spirit of the classical Liouville theorem
on Hamiltonian dynamics.
Let  $t_0 \in \R$.
By the definition of 
$\Psi(t, \tau)$, Lemma \ref{LEM:meas1}, and Proposition \ref{PROP:energy},  we have 
\begin{align*}
\frac{d}{dt}\rho_{s, N, r, t}&  (\Psi_N(t)(A))\bigg|_{t = t_0}\notag\\
& = \frac{d}{dt}\rho_{s, N, r, t_0+t} \big(\Psi_N(t_0+ t, t_0)(\Psi_N(t_0)(A))\big)\bigg|_{t = 0}\notag \\
& = \ft Z_{s, N, r}^{-1}
\frac{d}{dt} \int_{\Psi_N(t_0) (A)} \ind_{\{ \| v\|_{L^2 } \leq r\}} 
e^{- \EE_{N, t_0+t}( \P_{\leq N} \Psi_N(t_0+ t, t_0) (v))} d L_N \otimes d\mu_{s, N}^\perp \bigg|_{t = 0} \notag\\
& = -  
 \int_{\Psi_N(t_0)( A)} 
\frac{d}{dt} \EE_{N, t_0+t}\big( \P_{\leq N} \Psi_N(t_0 + t, t_0) (v)\big)\bigg|_{t = 0}
 d \rho_{s, N, r, t_0}\\
& \leq C_r
  \rho_{s, N, r, t_0}(\Psi_N(t_0) (A)).
\end{align*}

\noi
This proves \eqref{meas2}.
The second estimate \eqref{meas3} follows from a direct integration of \eqref{meas2}.
\end{proof}

As in \cite{OTz1}, 
we can upgrade 
 Lemma \ref{LEM:meas2}
to the untruncated measure $\rho_{s, r, t}$.

\begin{lemma}\label{LEM:meas4}
Let $ \frac 12< s < 1$.
Then, given  $t \in \R$ and  $r > 0$, 
 there exists $C = C(t, r) > 0$
such that 
\begin{align*}
\rho_{s,  r, t} (\Psi(t) (A)) \leq C\rho_{s,  r, t} ( A) 
\end{align*}

\noi
for  any measurable set 
$A  \subset L^2(\T)$.

\end{lemma}

This lemma follows from 
the approximation properties of $\Psi_N(t)$ to $\Psi(t)$ (Lemma \ref{LEM:approx})
and $\rho_{s, N, r, t}$ to $\rho_{s, r, t}$ (Lemma \ref{LEM:Gibbs}\,(ii))
along with some limiting argument.
See \cite[Lemma 6.10]{OTz1}
for the details of the proof.

Once we have Lemma \ref{LEM:meas4}, 
the proof of Theorem \ref{THM:1} follows
just as in \cite{OTz1}.
We present its proof for the convenience of readers.
Recall that in view of 
\eqref{nonlin2} and 
 Lemmas \ref{LEM:gauss}, 
that it suffices to prove that $\mu_s$
is quasi-invariant under $\Psi(t)$, i.e.~under the dynamics of \eqref{NLS5}.

Fix $t \in \R$. Let $A \subset L^2(\T)$ be a measurable set 
such that $\mu_s(A) = 0$.
Then, for any $r > 0$, we have 
\[\mu_{s, r}(A) = 0.\]

\noi
By the mutual absolute continuity 
 of $\mu_{s, r}$ and $\rho_{s, r, t}$, 
 we obtain
\[\rho_{s, r, t}(A) = 0\]

\noi
 for any $r > 0$.
Then, by Lemma \ref{LEM:meas4}, we have
\[\rho_{s, r, t}(\Psi(t) (A)) = 0.\]

\noi
By invoking the mutual absolute continuity 
 of $\mu_{s, r}$ and $\rho_{s, r, t}$ once again, 
 we have
\[\mu_{s, r}(\Psi(t) (A)) = 0.\]

\noi
Then, the dominated convergence theorem yields
\[\mu_{s}\big(\Psi(t) (A)\big) 
= \lim_{r \to \infty} 
\mu_{s, r}\big(\Psi(t) (A)\big) = 0.\]

\noi
This proves  Theorem \ref{THM:1},
assuming Proposition \ref{PROP:energy}.
In the next section, we implement an infinite iteration 
of normal form reductions and prove the improved energy estimate
(Proposition \ref{PROP:energy}).

\section{Proof of Proposition \ref{PROP:energy}: Normal form reductions}
\label{SEC:NF}

In this section, we present the proof of Proposition \ref{PROP:energy}
by implementing an infinite iteration scheme 
of normal form reductions.
This procedure allows us to construct an infinite sequences of correction terms
and thus build  the desired modified energies $\EE_{N, t}(v)$
and $\EE_{t}(v)$ in \eqref{energy6}.

Our main goal is to obtain an effective estimate on 
the growth of the $H^s$-norm of a solution $v$ to the truncated equation \eqref{NLS7}, 
independent of $N \in \N$.
For simplicity of presentation, however, we work on the equation \eqref{NLS5}
 without the frequency cutoff $\ind_{|n|\leq N}$
in the following.
We point out that the same normal form reductions
and estimates hold
for the truncated equation \eqref{NLS7}, uniformly in $N \in \N$, 
with  straightforward modifications:
(i) set 
$\ft v_n = 0$ for all $|n| >N$
and 
(ii) the multilinear forms for \eqref{NLS7}
are obtained by inserting 
the frequency cutoff $\ind_{|n|\leq N}$ in appropriate places.\footnote{Using the bi-trees introduced
in Subsection \ref{SUBSEC:tree} below, it follows
from \eqref{NLS7} that 
we simply need to insert the frequency cutoff $\ind_{|n^{(j)}|\leq N}$
on the parental frequency $n^{(j)}$ assigned to each non-terminal node $a \in \TT^0$.}
In the following, we introduce multilinear forms
such as 
$\NN_0^{(j)}$, $\NN_1^{(j)}$, and $\RR^{(j)}$
for the untruncated equation \eqref{NLS5}.
With a small modification, 
these multilinear forms give rise to
$\NN_{0, N}^{(j)}$, $\NN_{1, N}^{(j)}$, and $\RR_N^{(j)}$, 
$N \in \N$, 
for the truncated equation \eqref{NLS7},
appearing in Proposition \ref{PROP:energy}.
See Subsection \ref{SUBSEC:finite}.

\subsection{First few steps of normal form reductions}
\label{SUBSEC:41}

In the following, we describe the first few steps
of normal form reductions.
We keep the following discussion only at a formal level
since its purpose is to show the complexity of the problem
and the necessity of effective book-keeping notations 
that we introduce in the next subsection.
We will present the full procedure in Subsections \ref{SUBSEC:NF} and \ref{SUBSEC:error}.

Let $v \in C(\R; H^\infty(\T))$
be a global solution to \eqref{NLS5}.\footnote{While  we work 
with  \eqref{NLS5} without a frequency cutoff in the following, 
it follows from the uniform boundedness of the frequency truncation operator $\P_{\leq N}$
that our argument and estimates also hold for \eqref{NLS7}, uniformly in  $N \in \N$.
Noting that  any solution to \eqref{NLS7} (for some $N \in \N$) is smooth, 
the following computation can be easily justified for solutions to \eqref{NLS7}.
}
With $\phi(\bar n)$ and $\G(n)$ as in~\eqref{phiX} and~\eqref{Gam1}, 
we have\footnote{Recall our convention of using $v_n$ to denote the Fourier coefficient $\ft v_n$.}
\begin{align}
  \frac{d}{dt}\bigg(\frac{1}{2} \|v(t) \|_{H^s}^2\bigg)
& = - \Re i 
\sum_{n \in \Z} \sum_{\G(n)}
\jb{n}^{2s}
e^{ - i \phi(\bar n) t}   v_{n_1} \cj{ v_{n_2}}  v_{n_3} \cj{ v_n}\notag\\
&  = : \NN^{(1)}(v)(t).
\label{XN^0}
\end{align}

\noi
In view of Lemma \ref{LEM:phase} with \eqref{Gam1}, 
we have $|\phi(\bar n)| \geq 1 $
in the summation above.
Then, 
by performing a normal form reduction, 
namely, 
differentiating by parts as in  \eqref{P2}, 
we obtain
\begin{align}
\NN^{(1)}(v)(t)
& = \Re \frac{d}{dt} \bigg[
\sum_{n \in \Z}\sum_{\G(n)}
\frac{e^{ - i \phi(\bar n) t}}{ \phi(\bar n)} \jb{n}^{2s}  v_{n_1} \cj{ v_{n_2}}  v_{n_3} \cj{ v_n}\bigg]\notag \\
& \hphantom{X}
- \Re 
\sum_{n \in \Z}\sum_{\G(n)}
\frac{e^{ - i \phi(\bar n) t}}{ \phi(\bar n)} \jb{n}^{2s} 
\dt \big( v_{n_1} \cj{ v_{n_2}}  v_{n_3} \cj{ v_n}\big) \notag\\
& = \Re \frac{d}{dt} \bigg[
\sum_{n \in \Z}\sum_{\G(n)}
\frac{e^{ - i \phi(\bar n) t}}{ \phi(\bar n)} \jb{n}^{2s}  v_{n_1} \cj{ v_{n_2}}  v_{n_3} \cj{ v_n}\bigg]\notag \\
& \hphantom{X}
-  \Re  
\sum_{n \in \Z}\sum_{\G(n)}
\frac{e^{ - i \phi(\bar n) t}}{ \phi(\bar n)} \jb{n}^{2s} 
\Big\{ \RR(v)_{n_1}
\cj{ v_{n_2}}  v_{n_3} \cj{ v_n} \notag\\
& \hphantom{XXXXX}
 +  v_{n_1} \cj{\RR(v)_{n_2}}
  v_{n_3} \cj{ v_n}
+  v_{n_1}
\cj{ v_{n_2}}  \RR(v)_{n_3} \cj{ v_n}
 +  v_{n_1}\cj{ v_{n_2}}  v_{n_3} 
 \cj{\RR(v)_{n}}\Big\}
 \notag\\
& \hphantom{X}
-  \Re  
\sum_{n \in \Z}\sum_{\G(n)}
\frac{e^{ - i \phi(\bar n) t}}{ \phi(\bar n)} \jb{n}^{2s} 
\Big\{ \NN(v)_{n_1}
\cj{ v_{n_2}}  v_{n_3} \cj{ v_n} \notag\\
& \hphantom{XXXXX}
 +  v_{n_1} \cj{\NN(v)_{n_2}}
  v_{n_3} \cj{ v_n}
+  v_{n_1}
\cj{ v_{n_2}}  \NN(v)_{n_3} \cj{ v_n}
 +  v_{n_1}\cj{ v_{n_2}}  v_{n_3} 
 \cj{\NN(v)_{n}}\Big\}
 \notag\\
& =: \dt \NN_0^{(2)}(v)(t) +  \RR^{(2)}(v)(t) + \NN^{(2)}(v)(t).
\label{XN^1}
\end{align}

\noi
In the second equality, we
applied the product rule 
and  used the equation \eqref{NLS5}
to replace  $\dt  v_{n_j}$ (and  $\dt  v_n$, respectively) by 
the resonant part $\RR(v)_{n_j}$ 
(and $\RR(v)_{n}$, respectively) and the non-resonant part $\NN(v)_{n_j}$
(and $\NN(v)_{n}$, respectively).

As we see below, 
we can estimate  the boundary term $\NN_0^{(2)}$
and the contribution  $\RR^{(2)}$ from the resonant part
in a straightforward manner.

\begin{lemma}
\label{LEM:N^2_0}
Let $\NN^{(2)}_0$ and  $\RR^{(2)}$
be as in \eqref{XN^1}.
Then, we have
\begin{align*} 
| \NN^{(2)}_0(v)| 
& \les   \|v\|_{L^2}^4,\\
| \RR^{(2)}(v)| 
& \les  \|v\|_{L^2}^6.
\end{align*}

\end{lemma}

\noi
See Lemma \ref{LEM:N^J+1_0}
(with $J = 1$)
for the proof.

It remains  to treat
 the last term $\NN^{(2)}$ in \eqref{XN^1}.
For an expository purpose, 
we only consider  the first term among the four terms in $\NN^{(2)}$
in the following.
A full consideration is given in Subsection \ref{SUBSEC:NF}
once we introduce proper notations in the next subsection.
With~\eqref{NLS5}, we have
\begin{align}
 \NN^{(2)}(v)(t)
&  =    \Re  i 
\sum_{n \in \Z}\sum_{\G(n)}
\sum_{\substack{n_1 = m_1 - m_2 + m_3\\n_1 \ne m_1, m_3}}
\frac{e^{ - i (\phi_1 + \phi_2) t}}{ \phi_1} \jb{n}^{2s} 
(  v_{m_1}\cj{ v_{m_2}}  v_{m_3} )
\cj{ v_{n_2}}  v_{n_3} \cj{ v_n}, 
\label{XN^1_2}
\end{align}

\noi
where 
$\phi_1 = \phi(\bar n)$ and 
$\phi_2 = \phi(m_1, m_2, m_3, n_1) = m_1^4 - m_2^4 + m_3^4 - n_1^4$
denote the phase functions
from the first and second ``generations''.\footnote{In the next subsection, 
we make a precise definition of what we mean by ``generation''.}
It turns out  that we can not establish a direct 6-linear estimate
on  $\NN^{(2)}$ in  \eqref{XN^1_2}.

We divide the frequency space in \eqref{XN^1_2} into 
\begin{equation} \label{C1}
C_1 = \big\{ |\phi_1 + \phi_2| \leq 6^3\big\} 
\end{equation}

\noi
and its complement $C_1^c$.\footnote{Clearly, the number $6^3$ in \eqref{C1} 
does not make any difference at this point.
However, we insert it to match with \eqref{Cj}.
See also \eqref{C2}.}
We then write $\NN^{(2)}$ as 
\begin{equation}
\NN^{(2)} = \NN^{(2)}_1 + \NN_2^{(2)},
\label{N2}
\end{equation}

\noi
where $\NN^{(2)}_1$ is the restriction of $\NN^{(2)}$
onto $C_1$
and
$\NN_2^{(2)} := \NN^{(2)} - \NN^{(2)}_1$.
On the one hand, 
we can estimate the contribution $\NN^{(2)}_1$ from $C_1$
in an effective manner (Lemma \ref{LEM:N^J+1_1} with $J = 1$)
thanks to the frequency restriction on $C_1$.
On the other hand, 
$\NN_2^{(2)}$ can not be handled as  it is and 
thus we apply the second step of  normal form reductions
to $\NN^{(2)}_2$. 
After differentiation by parts with \eqref{NLS5} as in \eqref{XN^1}, 
we arrive at 
\begin{align*}
\NN^{(2)}_2 (v)
& = \dt \NN^{(3)}_0(v) + \RR^{(3)}(v) + \NN^{(3)}(v), 
\end{align*}

\noi
where 
$ \NN^{(3)}_0$ is a  6-linear form
and 
$\RR^{(3)}$ and  $\NN^{(3)}$ 
are 8-linear forms, corresponding to 
the contributions
from the resonant part $\RR(v)$ and the non-resonant part
$\NN(v)$ upon the substitution of \eqref{NLS5}.
See \eqref{N^2} below for the precise computation.

As in the previous step, 
we can estimate  $\NN^{(3)}_0$ and $\RR^{(3)}$ in a straightforward manner
(Lemma~\ref{LEM:N^J+1_0} with $J = 2$).
On the other hand, 
we can not estimate $\NN^{(3)}$ as it is
and hence we need to split it as 
\begin{equation} 
\NN^{(3)} = \NN^{(3)}_1 + \NN^{(3)}_2,
\label{N3}
\end{equation}

\noi
where $\NN^{(3)}_1$ is the restriction of $\NN^{(3)}$
onto 
\begin{equation} \label{C2}
C_2 = \big\{ |\phi_1 + \phi_2 + \phi_3| \le 8^3\big\} 
\end{equation}

\noi
and 
$\NN_2^{(3)} := \NN^{(3)} - \NN^{(3)}_1$.
Here, $\phi_j$, $j = 1, 2, 3$, denotes
the phase function from the $j$th ``generation''.
As we see below, 
 $\NN^{(3)}_1$ satisfies a good 8-linear estimate (Lemma \ref{LEM:N^J+1_1} with $J = 2$) 
thanks to the frequency restriction on $C_2$.
On the other hand, 
$\NN_2^{(3)}$ can not be handled as  it is and 
thus we apply the third step of  normal form reductions
to $\NN^{(3)}_2$. 
In this way, we iterate  normal form reductions in an indefinite manner.

As we iteratively apply normal form reductions, 
the degrees of the multilinear terms increase linearly.
After $J$ steps, 
we obtain 
the multilinear terms 
$\NN_0^{(J+1)}$
of degree $2J+2$ and $\RR^{(J+1)}$ and $\NN^{(J+1)}$
of degree $2J+4$.
See \eqref{N^J+1}.
As in the first and second steps described above, 
we also divide $\NN^{(J+1)}$ into ``good'' and ``bad'' parts
and apply another normal form reduction
to the bad part of degree $2J+4$, where time differentiation
can fall on any of the $2J+4$ factors.
An easy computation shows that the number of terms 
grows factorially (see~\eqref{cj1}) 
and hence we need to introduce an effective way to handle this combinatorial complexity.
In the next subsection, 
we introduce indexing notation by bi-trees,
which allows us to denote a factorially growing number of multilinear terms
in a concise manner.

\subsection{Notations: index by ordered bi-trees}
\label{SUBSEC:tree}

In \cite{GKO}, the first author with Guo and Kwon implemented
an infinite iteration of normal form reductions
to study the cubic NLS on $\T$,
where differentiation by parts was applied to the evolution equation satisfied by 
the interaction representation.
In \cite{GKO}, (ternary) trees and ordered trees played 
an important role 
for indexing various multilinear terms and frequencies arising in the general steps of the Poincar\'e-Dulac normal form reductions.

Our main goal here is to implement an infinite iteration scheme
of normal form reduction applied to the $H^s$-energy functional\footnote{More precisely, 
to the evolution equation satisfied by the $H^s$-energy functional.}
 $ \| v(t)\|_{H^s}^2$, as we saw above.
 In particular,  we need tree-like structures that grow in two directions.
For this purpose,  we introduce the notion of  bi-trees and ordered bi-trees
in the following.
Once we replace trees and ordered trees
by bi-trees and ordered bi-trees,
other related notions can be defined
in a similar manner as in \cite{GKO}
with certain differences to be noted.

\begin{definition} \label{DEF:tree1} \rm
Given a partially ordered set $\TT$ with partial order $\leq$, 
we say that $b \in \TT$ 
with $b \leq a$ and $b \ne a$
is a child of $a \in \TT$,
if  $b\leq c \leq a$ implies
either $c = a$ or $c = b$.
If the latter condition holds, we also say that $a$ is the parent of $b$.

\end{definition}

As in \cite{GKO},
our trees in this paper 
refer to a particular subclass of usual trees with the following properties.

\begin{definition} \label{DEF:tree2} \rm
(i) A tree $\TT $ is a finite partially ordered set satisfying
the following properties:
\begin{enumerate}

\item[(a)] Let $a_1, a_2, a_3, a_4 \in \TT$.
If $a_4 \leq a_2 \leq a_1$ and  
$a_4 \leq a_3 \leq a_1$, then we have $a_2\leq a_3$ or $a_3 \leq a_2$,

\item[(b)]
A node $a\in \TT$ is called terminal, if it has no child.
A non-terminal node $a\in \TT$ is a node 
with  exactly three ordered\footnote{For example, 
we simply label the three children as $a_1, a_2$, and $a_3$
 by moving from left to right in the planar graphical representation of the tree $\TT$.
 As we see below, we assign the Fourier coefficients of the interaction representation $v$ at $a_{1}$ and $a_{3}$, 
 while we assign the complex conjugate of
 the Fourier coefficients of $v$ at the second child $a_{2}$.} children denoted by $a_1, a_2$, and $a_3$,

\item[(c)] There exists a maximal element $r \in \TT$ (called the root node) such that $a \leq r$ for all $a \in \TT$,

\item[(d)] $\TT$ consists of the disjoint union of $\TT^0$ and $\TT^\infty$,
where $\TT^0$ and $\TT^\infty$
denote  the collections of non-terminal nodes and terminal nodes, respectively.
\end{enumerate}

\smallskip

\noi
(ii) A {\it bi-tree} $\TT = \TT_1 \cup \TT_2$ is 
a disjoint union of two trees $\TT_1$ and $\TT_2$,
where the root nodes $r_j$ of $\TT_j$, $j = 1, 2$,  are joined by an edge.
A bi-tree
$\TT$ consists of the disjoint union of $\TT^0$ and $\TT^\infty$,
where $\TT^0$ and $\TT^\infty$
denote  the collections of non-terminal nodes and terminal nodes, respectively.
By convention, we assume that the root node $r_1$ of the tree $\TT_1$ is non-terminal,
while the root node $r_2$ of the tree $\TT_2$ may be terminal.

\smallskip

\noi
(iii) Given a bi-tree $\TT = \TT_1 \cup \TT_2$, 
we define a projection $\Pi_j$, $j = 1, 2$, onto  a tree
by setting 
\begin{align}
\Pi_j(\TT) = \TT_j.
\label{proj1}
\end{align}

\noi
In Figure \ref{FIG:1}, 
$\Pi_1(\TT)$ corresponds to the tree on the left under the root node $r_1$, 
while 
$\Pi_2(\TT)$ corresponds to the tree on the right under the root node $r_2$.

\end{definition}

Note that the number $|\TT|$ of nodes in a bi-tree $\TT$ is $3j+2$ for some $j \in \mathbb{N}$,
where $|\TT^0| = j$ and $|\TT^\infty| = 2j + 2$.
Let us denote  the collection of trees in the $j$th generation 
(namely,  with $j$ parental nodes) by $BT(j)$, i.e.
\begin{equation*}
BT(j) := \{ \TT : \TT \text{ is a bi-tree with } |\TT| = 3j+2 \}.
\end{equation*}

\begin{figure}[h]
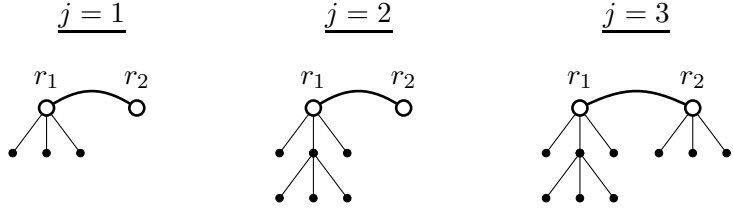

\[ \<T1>  
\qquad\qquad
\<T2> 
\qquad  \qquad
\<T3>\]

\caption{Examples of bi-trees of the $j$th generation, $j = 1, 2, 3$.}
\label{FIG:1}
\end{figure}

\vspace{3mm}

Next, we introduce the  notion of ordered bi-trees,
for which we  keep track of how a bi-tree ``grew''
into a given shape.

\begin{definition} \label{DEF:tree3} \rm
(i) We say that a sequence $\{ \TT_j\}_{j = 1}^J$  is a chronicle of $J$ generations, 
if 
\begin{enumerate}
\item[(a)] $\TT_j \in BT(j)$ for each $j = 1, \dots, J$,
\item[(b)]  $\TT_{j+1}$ is obtained by changing one of the terminal
nodes in $\TT_j$ into a non-terminal node (with three children), $j = 1, \dots, J - 1$.
\end{enumerate}

\noi
Given a chronicle $\{ \TT_j\}_{j = 1}^J$ of $J$ generations,  
we refer to $\TT_J$ as an {\it ordered bi-tree} of the $J$th generation.
We denote the collection of the ordered trees of the $J$th generation
by $\mathfrak{BT}(J)$.
Note that the cardinality of $\mathfrak{BT}(J)$ is given by 
$ |\mathfrak{BT}(1)| = 1$ and 
\begin{equation} 
\label{cj1}
 |\mathfrak{BT}(J)| = 4\cdot 6 \cdot 8 \cdot \cdots \cdot 2J 
 = 2^{J-1}   \cdot J!=: c_J,
 \quad J \geq 2.
 \end{equation}

\smallskip

\noi
(ii) Given an ordered bi-tree $\TT_J \in \mathfrak{BT}(J)$ as above, 
we define  projections $\pi_j$, $j = 1, \dots, J-1$, 
onto the previous generations
by setting
\begin{align*}
 \pi_j(\TT_J) = \TT_j \in \mathfrak{BT}(j).
\end{align*}

\end{definition}

We stress that the notion of ordered bi-trees comes with associated chronicles.
For example, 
given two ordered bi-trees $\TT_J$ and $\wt{\TT}_J$
of the $J$th generation, 
it may happen that $\TT_J = \wt{\TT}_J$ as bi-trees (namely as planar graphs) 
according to Definition \ref{DEF:tree2},
while $\TT_J \ne \wt{\TT}_J$ as ordered bi-trees according to Definition \ref{DEF:tree3}.
In the following, when we refer to an ordered bi-tree $\TT_J$ of the $J$th generation, 
it is understood that there is an underlying chronicle $\{ \TT_j\}_{j = 1}^J$.

\smallskip

Given a bi-tree $\TT$, 
we associate each terminal node $a \in \TT^\infty$ with the Fourier coefficient (or its complex conjugate) of the interaction representation 
$v$ and sum over all possible frequency assignments.
In order to do this, we introduce index functions, 
assigning integers to {\it all} the nodes in $\TT$ in a consistent manner.

\begin{definition} \label{DEF:tree4} \rm
(i) Given  a bi-tree $\TT = \TT_1\cup \TT_2$, 
we define an index function ${\bf n}: \TT \to \mathbb{Z}$ such that
\begin{itemize}

\item[(a)] $n_{r_1} = n_{r_2}$, where $r_j$ is the root node of the tree $\TT_j$, $j = 1, 2$,

\item[(b)] $n_a = n_{a_1} - n_{a_2} + n_{a_3}$ for $a \in \TT^0$,
where $a_1, a_2$, and $a_3$ denote the children of $a$,

\item[(c)] $\{n_a, n_{a_2}\} \cap \{n_{a_1}, n_{a_3}\} = \emptyset$ for $a \in \TT^0$,

\end{itemize}

\noi
where  we identified ${\bf n}: \TT \to \mathbb{Z}$ 
with $\{n_a \}_{a\in \TT} \in \mathbb{Z}^\TT$. 
We use 
$\mathfrak{N}(\TT) \subset \mathbb{Z}^\TT$ to denote the collection of such index functions ${\bf n}$
on $\TT$.

Given $N \in \mathbb{N}$, 
we define a subcollection 
$\mathfrak{N}_N(\TT) \subset \mathfrak{N}(\TT)$ 
by imposing $|n_a| \leq N$ for any $a \in \TT$.
We also define
$\mathfrak{N}_N^0(\TT) \subset \mathfrak{N}(\TT)$ 
by imposing $|n_a| \leq N$ for any non-terminal nodes $a \in \TT^0$.

\smallskip

\noi
(ii) Given a tree $\TT$, we also define 
 an index function ${\bf n}: \TT \to \mathbb{Z}$ 
 by omitting the condition (a)
 and denote by 
 $\mathfrak{N}(\TT) \subset \mathbb{Z}^\TT$  the collection of index functions ${\bf n}$
 on $\TT$.

\end{definition}

\begin{remark} \label{REM:terminal}
\rm 
(i) In view of the consistency condition (a), 
we can refer to $n_{r_1} = n_{r_2}$
as the frequency  at the root node without ambiguity.
We shall  simply denote it by $n_r$ in the following.

\smallskip

\noi
(ii) 
Just like  index functions  for (ordered) trees considered in \cite{GKO}, 
an index function ${\bf n} = \{n_a\}_{a\in\TT}$ for a bi-tree $\TT$ is completely determined
once we specify the values $n_a \in \Z$ for all the  terminal nodes $a \in \TT^\infty$.
An index function $\bn$
for a bi-tree $\TT = \TT_1 \cup \TT_2$
is basically a pair $(\bn_1, \bn_2)$
of index functions $\bn_j$ for the trees $\TT_j$, $j = 1, 2$, (omitting
the non-resonance condition in \cite[Definition 3.5 (iii)]{GKO}),
satisfying the consistency condition (a): $n_{r_1} = n_{r_2}$.

\smallskip

\noi
(iii) Given a bi-tree $\TT \in \mathfrak{BT}(J)$ and $n \in \Z$, consider
 the summation of all possible frequency assignments
 $\{ \bn \in \mathfrak{N}(\TT): n_r = n\}$.
While $|\TT^\infty| = 2J + 2$, 
there are $2J$ free variables in this summation.
Namely, the condition $n_r = n$ reduces two summation variables.
It is easy to see this by separately considering
the cases $\Pi_2(\TT) = \{r_2\}$
and $\Pi_2(\TT) \ne \{r_2\}$.
\end{remark}

\medskip

Given an ordered bi-tree 
$\TT_J$ of the $J$th generation with a chronicle $\{ \TT_j\}_{j = 1}^J$ 
and associated index functions ${\bf n} \in \mathfrak{N}(\TT_J)$,
we would like to keep track of the  ``generations'' of frequencies.
In the following,  we use superscripts to denote such generations of frequencies.

Fix ${\bf n} \in \mathfrak{N}(\TT_J)$.
Consider $\TT_1$ of the first generation.
Its nodes consist of the two root nodes $r_1$, $r_2$, 
and the children $r_{11}, r_{12}, $ and $r_{13}$ of the first root node $r_1$. 
See Figure \ref{FIG:1}.
We define the first generation of frequencies by
\[\big(n^{(1)}, n^{(1)}_1, n^{(1)}_2, n^{(1)}_3\big) :=(n_{r_1}, n_{r_{11}}, n_{r_{12}}, n_{r_{13}}).\]

\noi
From Definition \ref{DEF:tree4}, we have
\begin{equation*}
n^{(1)}  = n_{r_2}, \quad  n^{(1)} = n^{(1)}_1 - n^{(1)}_2 + n^{(1)}_3, \quad n^{(1)}_2\ne n^{(1)}_1, n^{(1)}_3.
\end{equation*}

Next, we construct  an ordered bi-tree $\TT_2$ of the second generation 
from $\TT_1$ by
changing one of its terminal nodes $a \in \TT^\infty_1 = \{ r_2, r_{11}, r_{12}, r_{13}\}$ 
into a non-terminal node.
Then, we define
the second generation of frequencies by setting
\[\big(n^{(2)}, n^{(2)}_1, n^{(2)}_2, n^{(2)}_3\big) :=(n_a, n_{a_1}, n_{a_2}, n_{a_3}).\]

\noi
Note that  we have $n^{(2)} = n^{(1)}$ or $n_k^{(1)}$ for some $k \in \{1, 2, 3\}$, 
\begin{equation*}
 n^{(2)} = n^{(2)}_1 - n^{(2)}_2 + n^{(2)}_3, \quad n^{(2)}_2\ne n^{(2)}_1, n^{(2)}_3,
\end{equation*}

\noi
where the last identities follow from Definition \ref{DEF:tree4}.
This extension of $\TT_1 \in \mathfrak{BT}(1)$ to $\TT_2\in \mathfrak{BT}(2)$
corresponds to introducing a new set of frequencies
after the first differentiation by parts,
where  the time derivative may fall on any of $v_n$ and $v_{n_j}$, $j = 1, 2, 3$.\footnote{The complex conjugate signs on  $v_n$ and $v_{n_j}$ do not play any significant role.
Hereafter,  we drop the complex conjugate sign, 
when it does not play any important role.}

In general, we construct  an ordered bi-tree $\TT_j$ 
of the $j$th generation from $\TT_{j-1}$ by
changing one of its terminal nodes $a  \in \TT^\infty_{j-1}$
into a non-terminal node.
Then, we define
the $j$th generation of frequencies by
\[\big(n^{(j)}, n^{(j)}_1, n^{(j)}_2, n^{(j)}_3\big) :=(n_a, n_{a_1}, n_{a_2}, n_{a_3}).\]

\noi
As before, it follows from Definition \ref{DEF:tree4} that 
\begin{equation*} 
 n^{(j)} = n^{(j)}_1 - n^{(j)}_2 + n^{(j)}_3, \quad n^{(j)}_2\ne n^{(j)}_1, n^{(j)}_3.
\end{equation*}

\noi
 Given an ordered bi-tree $\TT$, we denote by $B_j = B_j(\TT)$
the set of all possible frequencies in the $j$th generation.

We denote by  $\phi_j$  the 
phase function 
for the frequencies introduced at the $j$th generation:
\begin{align*}
\phi_j & =  \phi_j \big( n^{(j)},  n^{(j)}_1, n^{(j)}_2, n^{(j)}_3\big)
:=   \big(n_1^{(j)}\big)^4 - \big(n_2^{(j)}\big)^4 + \big(n_3^{(j)}\big)^4
- \big(n^{(j)}\big)^4.
\end{align*}

\noi
Note that we have $|\phi_1| \geq 1$ in view of Definition \ref{DEF:tree4} and Lemma \ref{LEM:phase}.
We also denote by $\mu_j$
 the  phase function corresponding to the usual cubic NLS (at the $j$th generation):
\begin{align*}
\mu_j & =  \mu_j \big( n^{(j)},  n^{(j)}_1, n^{(j)}_2, n^{(j)}_3\big)
:=   \big(n_1^{(j)}\big)^2 - \big(n_2^{(j)}\big)^2 + \big(n_3^{(j)}\big)^2
- \big(n^{(j)}\big)^2\notag \\
& =-2 \big(n^{(j)} - n_1^{(j)}\big) \big(n^{(j)} - n_3^{(j)}\big).
\end{align*}

\noi
Then, by Lemma \ref{LEM:phase}, we have 
\begin{align}
|\phi_j| \sim (n^{(j)}_\text{max})^2
\cdot| \big(n^{(j)} - n_1^{(j)}\big) \big(n^{(j)} - n_3^{(j)}\big)|
\sim (n^{(j)}_\text{max})^2 \cdot |\mu_j|,
\label{MU2}
\end{align}

\noi
where $n^{(j)}_\text{max}: = \max\big(|n^{(j)}|, 
|n_1^{(j)}|, |n_2^{(j)}|, |n_3^{(j)}| \big)$.

 Given  an ordered bi-tree $\TT \in \mathfrak{BT}(J)$ for some $J \in \N$, 
define $C_j \subset \mathfrak{N}(\TT)$ by 
\begin{equation} \label{Cj}
C_j = \big\{ |\wt{\phi}_{j+1}| \leq  (2j+4)^3\big\} , 
\end{equation}

\noi
where $\wt{\phi}_j$  is defined by 
\begin{align}
 \wt{\phi}_j = \sum_{k = 1}^j \phi_k.
\label{Cj2}
\end{align}

\noi
In Subsection \ref{SUBSEC:NF}, 
we perform normal form reductions in an iterative manner.
At each step, 
we divide multilinear forms into
``nearly resonant'' part (corresponding to the frequencies belonging to $C_j$) and highly non-resonant part
(corresponding to the frequencies belonging to $C_j^c$)
and apply a normal form reduction only to the highly non-resonant part.

\subsection{Arithmetic lemma}
\label{SUBSEC:sum}

As we see in the next subsection, 
normal form reductions generate multilinear forms of higher and higher degrees,
where we need to sum over all possible
ordered bi-trees in  $\mathfrak{BT}(J)$.
The main issue is then to control   the rapidly growing
 cardinality $c_J = |\mathfrak{BT}(J)| $  defined in~\eqref{cj1}.
On the one hand,  we utilize the divisor counting estimate (see \eqref{divisor} below)
as in \cite{GKO}.
On the other hand, 
we split the argument into two parts.
The following lemma shows the heart of the matter
in the multilinear estimates presented in the next subsection.
This allows us to show that there is a sufficiently fast decay
at each step of normal form reductions.

\begin{lemma}\label{LEM:sum1}
Let $ s< 1$ and $J \in \N$.
Then, the following estimates hold:
\begin{align}
\textup{(i)}& \qquad  
\sup_{\TT_J \in \mathfrak{BT}(J)}
\sup_{n \in \Z} \sum_{\substack{{\bf n} \in \mathfrak{N}(\TT_J)\\{ n}_r = n}}
\ind_{\bigcap_{j = 1}^{J-1} C_j^c}
\frac{\jb{n}^{4s}}{|\phi_1|^{2}} 
\prod_{j = 2}^J \frac{1}{|\wt{\phi}_j|^2}
  \lesssim \frac{1}{\prod_{j = 2}^J(2j+2)^{3-}}, 
\label{sum1}\\
\textup{(ii)} & 
\qquad 
\sup_{\TT_{J+1} \in \mathfrak{BT}(J+1)}
\sup_{n \in \Z} \sum_{\substack{{\bf n} \in \mathfrak{N}(\TT_{J+1})\\{ n}_r = n}}
\ind_{(\bigcap_{j = 1}^{J-1} C_j^c) \cap C_J}
\frac{\jb{n}^{4s}}{|\phi_1|^{2}} 
\prod_{j = 2}^J \frac{1}{|\wt{\phi}_j|^2}
  \lesssim \frac{J^{3+}}{\prod_{j = 2}^J(2j+2)^{3-}}.
\label{sum2}
\end{align}

\end{lemma}

Before proceeding further, let us 
recall the following arithmetic fact \cite{HW}.
Given $n \in \N$, the number $d(n)$ of the divisors of $n$
satisfies
\begin{align}
d(n) \leq C_\dl n^\dl
\label{divisor}
\end{align}

\noi
for any $\dl > 0$.	
This divisor counting estimate will be used iteratively in the following proof.

\begin{proof}[Proof of Lemma \ref{LEM:sum1}]
(i)
We first consider the case $J = 1$.
In this case, 
from Lemma \ref{LEM:phase} with $s \leq 1$,  we have
\begin{align}
\text{LHS of }\eqref{sum1}
\leq
\sup_{n \in \Z} \sum_{\substack{n_1, n_3 \in \Z\\n_1, n_3 \ne n\\|\phi_1|\geq 1}}
 \frac{\jb{n}^{4s}}{|\phi_1|^{2}}
\les \sup_{n \in \Z}  \sum_{\substack{n_1, n_3 \in \Z\\n_1, n_3 \ne n}}
\frac{1}{|(n-n_1)(n-n_3)|^{2}}
\les 1.
\label{sum1a}
\end{align}

Next,  we  consider the case $J \geq 2$.
Fix $\TT_J \in \mathfrak{BT}(J)$.
For simplicity of notations, 
we drop the supremum 
over $\TT_J \in \mathfrak{BT}(J)$ in the following
with the understanding that the implicit constants are independent
of $\TT_J \in \mathfrak{BT}(J)$.
A similar comment applies to the proof of the estimate \eqref{sum2} presented in (ii) below.

The main idea is to apply the divisor counting argument in an iterative manner.
It follows from the divisor counting estimate \eqref{divisor} 
with the factorization of $\phi_j$ (Lemma \ref{LEM:phase}) that 
for fixed  $n^{(j)}$ and $\phi_j$,
there are at most $O(|\phi_j|^{0+})$ many choices for $n^{(j)}_1$, $n^{(j)}_2$, and $n^{(j)}_3$
on $B_j$.
Also, note that $\phi_j$ is determined by $\wt{\phi}_1, \dots, \wt{\phi}_j$
and 
\begin{equation} \label{mujj}
|\phi_j| \leq \max(|\wt{\phi}_{j-1}|, |\wt{\phi}_j|).
\end{equation}

\noi
since $\phi_j = \wt{\phi}_{j}-\wt{\phi}_{j-1}$.
In the following, we apply the divisor counting argument
to sum over the frequencies in $B_J$, $B_{J-1}$, \dots, $B_2$.
From Definition \ref{DEF:tree3} (ii) and \eqref{Cj}, we have
\begin{align*}
\text{LHS of }\eqref{sum1}
& = 
 \sup_{n\in \Z}
\sum_{\substack{{\bf n} \in \mathfrak{N}(\pi_1(\TT_J))\\{ n}_r = n\\|\phi_1| \geq 1}}
\frac{\jb{n}^{4s}}{|\phi_1|^{2}} 
\sum_{\substack{\psi_2 \in \Z\\ |\psi_2| >  6^3 }} \sum_{\substack{B_2\\ \wt \phi_2 = \psi_2}} 
\frac{1}{|\psi_2|^2}
\, \cdots 
\sum_{\substack{\psi_J \in \Z\\ |\psi_J| >  (2J+2)^3 }} \sum_{\substack{B_J\\ \wt \phi_J = \psi_J}} 
\frac{1}{|\psi_J|^2}
\notag \\
\intertext{By applying the divisor counting argument in $B_J$ with \eqref{mujj}, 
we have}
& \les 
 \sup_{n \in \Z}
\sum_{\substack{{\bf n} \in \mathfrak{N}(\pi_1(\TT_J))\\{ n}_r = n\\|\phi_1| \geq 1}}
\frac{\jb{n}^{4s}}{|\phi_1|^{2}} 
\sum_{\substack{\psi_2 \in \Z\\ |\psi_2| >  6^3 }} \sum_{\substack{B_2\\ \wt \phi_2 = \psi_2}} 
\frac{1}{|\psi_2|^2} \, \cdots \notag\\
& \hphantom{XX}
\sum_{\substack{\psi_{J-1} \in \Z\\ |\psi_{J-1}| >  (2J)^3 }} \sum_{\substack{B_{J-1}\\ \wt \phi_{J-1} = \psi_{J-1}}} 
\frac{1}{|\psi_{J-1}|^2}
\sum_{\substack{\psi_J \in \Z\\ |\psi_J| >  (2J+2)^3 }} 
\frac{1}{|\psi_J|^{2}} |\psi_{J-1}|^{0+} |\psi_J|^{0+}
\intertext{By iteratively applying the divisor counting argument in $B_{J-1}$, \dots, $B_2$, 
we have}
& \les
 \sup_{n \in \Z}
\sum_{\substack{{\bf n} \in \mathfrak{N}(\pi_1(\TT_J))\\{ n}_r = n\\|\phi_1| \geq 1}}
\frac{\jb{n}^{4s}}{|\phi_1|^{2-}} 
\sum_{\substack{\psi_2, \dots, \psi_J \in \Z\\ |\psi_j| >  (2j+2)^3\\j = 2, \dots, J}}
\prod_{j = 2}^J\frac{1}{|\psi_j|^{2-}}\notag\\
& \lesssim \frac{1}{\prod_{j = 2}^J(2j+2)^{3-}},
\end{align*}

\noi
where the last inequality follows from  
\eqref{sum1a}.

\medskip

\noi
(ii) Fix $\TT_{J+1} \in \mathfrak{BT}(J+1)$.
We proceed with the divisor counting argument as in (i).
From~\eqref{Cj}, we have 
$|\phi_{J+1}| \les |\wt{\phi}_J| + J^3$ on $C_J$ and thus
for fixed  $n^{(J+1)}$ and $\phi_{J+1}$,
there are at most  
$O(J^{0+} |\wt \phi_{J}|^{0+})$ 
many choices for $n^{(J+1)}_1$, $n^{(J+1)}_2$, and $n^{(J+1)}_3$
on $B_{J+1}$.
Also, on $C_J$, 
there are at most $O\big(J^3\big)$ many choices
for $\wt{\phi}_{J+1}$.
Hence,
 for fixed $\wt{\phi}_J$,
there are also at most $O\big(J^3\big)$ many choices 
 for $\phi_{J+1} = \wt{\phi}_{J+1} - \wt{\phi}_J$
 on $C_J$.
Then, the contribution to \eqref{sum2} in this case is estimate by 
\begin{align*}
\text{LHS of } \eqref{sum2}
& = 
 \sup_{n \in \Z}
\sum_{\substack{{\bf n} \in \mathfrak{N}(\pi_1(\TT_{J+1}))\\{ n}_r = n\\|\phi_1| \geq 1}}
\frac{\jb{n}^{4s}}{|\phi_1|^{2}} 
\sum_{\substack{\psi_2 \in \Z\\ |\psi_2| >  6^3 }} \sum_{\substack{B_2\\ \wt \phi_2 = \psi_2}} 
\frac{1}{|\psi_2|^2}
\, \cdots \\
& \hphantom{XX}
\sum_{\substack{\psi_J \in \Z\\ |\psi_J| >  (2J+2)^3 }} \sum_{\substack{B_J\\ \wt \phi_J = \psi_J}} 
\frac{1}{|\psi_J|^2}
\sum_{\substack{B_{J+1}\\ |\phi_{J+1}+\psi_J| \leq  (2J+4)^3}} 1 
\notag \\
\intertext{By applying the divisor counting argument in $B_{J+1}$,
we have}
&\les  J^{3+} \sup_{n \in \Z}
\sum_{\substack{{\bf n} \in \mathfrak{N}(\pi_1(\TT_{J+1}))\\{ n}_r = n\\|\phi_1| \geq 1}}
\frac{\jb{n}^{4s}}{|\phi_1|^{2}} 
\sum_{\substack{\psi_2 \in \Z\\ |\psi_2| >  6^3 }} \sum_{\substack{B_2\\ \wt \phi_2 = \psi_2}} 
\frac{1}{|\psi_2|^2}
\, \cdots 
\\
& \hphantom{XX}
\sum_{\substack{\psi_J \in \Z\\ |\psi_J| >  (2J+2)^3 }} \sum_{\substack{B_J\\ \wt \phi_J = \psi_J}} 
\frac{1}{|\psi_J|^2}
  |\psi_J|^{0+}
\intertext{By iteratively applying the divisor counting argument in $B_{J}$, \dots, $B_2$
and then applying~\eqref{sum1a}, 
we have}
& \les 
J^{3+}
\sup_{n \in \Z}
\sum_{\substack{{\bf n} \in \mathfrak{N}(\pi_1(\TT_{J+1}))\\{ n}_r = n\\|\phi_1|\geq 1} }
\frac{\jb{n}^{4s}}{|\phi_1|^{2-}}
\sum_{\substack{\psi_J \in \Z\\  |\psi_J|>(2j+2)^3\\ j = 2, \dots, J} }
\prod_{j = 2}^J \frac{1}{|\psi_j|^{2-}}
 \notag\\
&  \lesssim \frac{J^{3+}
}{\prod_{j = 2}^J(2j+2)^{3-}}.
\end{align*}

\noi
This proves \eqref{sum2}.
\end{proof}

\begin{remark} \rm
In \cite{GKO}, the authors applied the divisor counting argument even to the frequencies of the first generation.
On the other hand,  we did not apply
the divisor counting argument to the frequencies of the first generation
in the proof of Lemma \ref{LEM:sum1} above.
Instead,  we simply used \eqref{sum1a} to control  the first generation.
By using only the factor $\mu_1 = -2 (n^{(1)} - n_1^{(1)})(n^{(1)} - n_3^{(1)})$
(and not the entire $\phi_1$)
for the summation, 
\eqref{sum1a} 
allows us to exhibit the required smoothing
in Proposition \ref{PROP:energy}.
\end{remark}

\subsection{Normal form reductions}
\label{SUBSEC:NF}

With the notations introduced in Subsection \ref{SUBSEC:tree}, 
let us revisit the discussion in Subsection \ref{SUBSEC:41}
and then discuss the general $J$th step.
We first implement
a formal infinite iteration scheme of normal form reductions
without justifying switching of limits and summations.
We justify formal computations at the end of this subsection.
Let $v \in C(\R; H^\infty(\T))$
be a global solution to \eqref{NLS5}.
Using the notations introduced in Subsection \ref{SUBSEC:tree}, 
we write \eqref{XN^0} as 
\begin{align*}
  \frac{d}{dt}\bigg(\frac{1}{2} \|v(t) \|_{H^s}^2\bigg)
& = 
- \Re i 
\sum_{\TT_1 \in \mathfrak{BT}(1)}
\sum_{{\bf n} \in \mathfrak{N}(\TT_1)} 
\jb{n_r}^{2s}
e^{  - i  \phi_1 t }  \prod_{a \in \TT^\infty_1} v_{n_{a}}
= : \NN^{(1)}(v)(t).
\end{align*}

\noi
By performing a normal form reduction, 
we then obtain
\begin{align}
\NN^{(1)}(v)(t)
& = 
 \Re \dt \bigg[
\sum_{\TT_1 \in \mathfrak{BT}(1)}
\sum_{{\bf n} \in \mathfrak{N}(\TT_1)}
\frac{ \jb{n_r}^{2s} e^{- i  \phi_1 t } }{\phi_1}
\prod_{a \in \TT_1^\infty} v_{n_{a}}
\bigg]\notag \\
& \hphantom{X}
- \Re 
\sum_{\TT_1 \in \mathfrak{BT}(1)}
\sum_{{\bf n} \in \mathfrak{N}(\TT_1)}
\frac{ \jb{n_r}^{2s} e^{- i  \phi_1 t } }{\phi_1}
\dt\bigg(\prod_{a \in \TT^\infty_1} v_{n_{a}}\bigg)
\notag\\
& =  \Re \dt \bigg[
\sum_{\TT_1 \in \mathfrak{BT}(1)}
\sum_{{\bf n} \in \mathfrak{N}(\TT_1)}
\frac{ \jb{n_r}^{2s} e^{- i  \phi_1 t } }{\phi_1}
\prod_{a \in \TT^\infty_1} v_{n_{a}}
\bigg]\notag \\
& \hphantom{X}
-   \Re 
\sum_{\TT_1 \in \mathfrak{BT}(1)}
\sum_{b\in \TT_1^\infty}
\sum_{{\bf n} \in \mathfrak{N}(\TT_1)}
\frac{ \jb{n_r}^{2s} e^{- i  \phi_1 t } }{\phi_1}
\RR(v)_{n_b}
\prod_{a \in \TT^\infty_1\setminus\{b\}} v_{n_{a}}\notag \\
& \hphantom{X}
+ \Re i 
\sum_{\TT_2 \in \mathfrak{BT}(2)}
\sum_{{\bf n} \in \mathfrak{N}(\TT_2)}
\frac{ \jb{n_r}^{2s} e^{- i  ( \phi_1+ \phi_2) t } }{\phi_1}
\prod_{a \in \TT^\infty_2} v_{n_{a}}
 \notag\\
& =: \dt \NN_0^{(2)}(v)(t) +  \RR^{(2)}(v)(t) + \NN^{(2)}(v)(t).
\label{N^1}
\end{align}

\noi
Compare \eqref{N^1} with \eqref{XN^1}.
In the second equality, we
applied the product rule 
and  used the equation \eqref{NLS5}
to replace  $\dt v_{n_b}$ by 
the resonant part $\RR(v)_{n_b}$ and the non-resonant part $\NN(v)_{n_b}$.
In substituting the non-resonant part $\NN(v)_{n_b}$, 
we turned  the terminal node $b \in \TT^\infty_1$
into a non-terminal node with three children $b_1, b_2,$ and $b_3$,
which corresponds to extending the tree $\TT_1 \in \mathfrak{BT}(1)$
(and ${\bf n }\in \mathfrak{N}(\TT_1)$)
to $\TT_2 \in \mathfrak{BT}(2)$
(and to  ${\bf n }\in \mathfrak{N}(\TT_2)$, respectively).

\begin{remark} \label{REM:tdepend} \rm

(i)
Strictly speaking, the  phase factor appearing in $\NN^{(2)}(v)$
may be $\phi_1 - \phi_2$
when the time derivative falls on the terms with the complex conjugate.
In the following, however, we simply write it as $\phi_1 + \phi_2$ since
it does not make any difference in our analysis.
Also, we often replace $\pm 1$ and $\pm i$ by $1$
for simplicity when they do not play an important  role.
Lastly, for notational simplicity, 
we drop  the real part symbol 
on  multilinear forms
with the understanding that  all the multilinear forms
appear with 
 the real part symbol.

\smallskip

\noi
(ii)
Due to the presence of 
$e^{-i \phi_1 t}$ 
in their definitions, 
the multilinear forms such as $\NN_0^{(2)}(v)$ 
are non-autonomous in $t$.
Therefore, strictly speaking, 
they should be denoted as  $\NN_0^{(2)}(t)(v(t))$. 
In the following, however, we establish nonlinear estimates
on these multilinear forms, uniformly in  $t \in \R$,
by simply using $|e^{-i \phi_1 t}| = 1$.
Hence, 
we simply suppress such $t$-dependence
when there is no confusion.
The same comment applies to other multilinear forms.
Note that this convention was already used in Proposition \ref{PROP:energy}.

It is  worthwhile to note that the multilinear forms introduced in this section 
are non-autonomous when they are expressed in terms 
of the interaction representation $v$, solving \eqref{NLS5}.
When they are expressed in terms of the original solution $u$ to \eqref{4NLS0}
(or $\wt u$ to \eqref{NLS4}), however, 
it is easy to see that these multilinear terms are indeed autonomous.

\end{remark}

Thanks to Lemma~\ref{LEM:sum1}, 
the terms  $\NN_0^{(2)}$  and   $\RR^{(2)}$ 
can be  estimated 
in a straightforward manner;
 see   Lemma~\ref{LEM:N^J+1_0} below.
We split
 $\NN^{(2)}$ as in \eqref{N2}.
As mentioned above, 
the good part  $\NN^{(2)}_1$ 
is handled 
in an effective manner (Lemma \ref{LEM:N^J+1_1})
thanks to the frequency restriction on $C_1$.
We then apply the second step of  normal form reductions
to $\NN^{(2)}_2$
and obtain
\begin{align}
\NN^{(2)}_2 (v)
& = \dt \bigg[\sum_{\TT_2 \in \mathfrak{BT}(2)}
\sum_{{\bf n} \in \mathfrak{N}(\TT_2)}
\ind_{C_1^c} \frac{\jb{n_r}^{2s}e^{- i( \phi_1 + \phi_2 )t } }{\phi_1(\phi_1+\phi_2)}
\prod_{a \in \TT^\infty_2} v_{n_{a}} \bigg]\notag \\
& \hphantom{X} -  
\sum_{\TT_2 \in \mathfrak{BT}(2)}
\sum_{b \in\TT^\infty_2} 
\sum_{{\bf n} \in \mathfrak{N}(\TT_2)}
\ind_{C_1^c}\frac{\jb{n_r}^{2s}e^{- i( \phi_1 + \phi_2)t } }{\phi_1(\phi_1+\phi_2)}
\, \RR(v)_{n_b}
\prod_{a \in \TT^\infty_2 \setminus \{b\}} v_{n_{a}} \notag \\
& \hphantom{X} 
- \sum_{\TT_3 \in \mathfrak{BT}(3)}\sum_{{\bf n} \in \mathfrak{N}(\TT_3)}
\ind_{C_1^c}\frac{\jb{n_r}^{2s}e^{- i( \phi_1 + \phi_2 +\phi_3)t } }{\phi_1(\phi_1+\phi_2)}
\,\prod_{a \in \TT^\infty_3} v_{n_{a}} \notag\\
& =: \dt \NN^{(3)}_0(v) + \RR^{(3)}(v) + \NN^{(3)}(v).
\label{N^2}
\end{align}

\noi
As in the previous step, 
we can estimate  $\NN^{(3)}_0$ and $\RR^{(3)}$ in a straightforward manner
(Lemma~\ref{LEM:N^J+1_0}), 
while we split $\NN^{(3)}$ 
into the good part $\NN^{(3)}_1$ and the bad part $\NN^{(3)}_2$ as in \eqref{N3}, 
where $\NN^{(3)}_1$ is the restriction of $\NN^{(3)}$
onto $C_2$ defined in \eqref{Cj}.
We then apply the third step of normal form reductions 
to the bad part $\NN^{(3)}_2$.
In this way, we iterate  normal form reductions in an indefinite manner.

After the $J$th step, we have
\begin{align} \label{N^J+1}
\NN^{(J)}_2 (v)
& = \dt \bigg[ 
\sum_{\TT_J \in \mathfrak{BT}(J)}
\sum_{{\bf n} \in \mathfrak{N}(\TT_J)}
\ind_{\bigcap_{j = 1}^{J-1} C_j^c}
\frac{\jb{n_r}^{2s}e^{- i \wt{\phi}_Jt } }{\prod_{j = 1}^J \wt{\phi}_j}
\, \prod_{a \in \TT^\infty_J} v_{n_{a}}
\bigg]\notag \\
& \hphantom{X} 
- \sum_{\TT_{J} \in \mathfrak{BT}(J)}
\sum_{b \in\TT^\infty_J} 
\sum_{{\bf n} \in \mathfrak{N}(\TT_J)}
\ind_{\bigcap_{j = 1}^{J-1} C_j^c}
\frac{\jb{n_r}^{2s}e^{- i \wt{\phi}_Jt } }{\prod_{j = 1}^J \wt{\phi}_j}
\, \RR(v)_{n_b}
\prod_{a \in \TT^\infty_J \setminus \{b\}} v_{n_{a}}  \notag \\
& \hphantom{X} 
- \sum_{\TT_{J+1} \in \mathfrak{BT}(J+1)}
\sum_{{\bf n} \in \mathfrak{N}(\TT_{J+1})}
\ind_{\bigcap_{j = 1}^{J-1} C_j^c}
\frac{\jb{n_r}^{2s}e^{- i \wt{\phi}_{J+1}t } }{\prod_{j = 1}^J \wt{\phi}_j}
\,\prod_{a \in \TT^\infty_{J+1}} v_{n_{a}} \notag\\
& =: \dt \NN^{(J+1)}_0 (v)+ \RR^{(J+1)}(v) + \NN^{(J+1)}(v),
\end{align}

\noi
where
$\wt{\phi}_J$ is  as in \eqref{Cj2}.
In the following, we first  estimate $\NN^{(J+1)}_0$ and  $\RR^{(J+1)}$
by  applying Cauchy-Schwarz inequality 
and then applying the divisor counting argument (Lemma \ref{LEM:sum1}).

\begin{lemma}\label{LEM:N^J+1_0}
Let $\NN^{(J+1)}_0$ and $\RR^{(J+1)}$ be as in \eqref{N^J+1}.
Then, for any $s < 1$, 
we have
\begin{align} 
| \NN^{(J+1)}_0(v)|  \lesssim 
 \frac{c_J}{\prod_{j = 2}^J(2j+2)^{\frac{3}{2}-}}
 \|v\|_{L^2}^{2J+2}, 
 \label{N^J+1_0-1}\\
 | \RR^{(J+1)}(v)|
 \lesssim 
 \frac{ (2J+2)\cdot c_J}{\prod_{j = 2}^J(2j+2)^{\frac{3}{2}-}}
\|v\|_{L^2}^{2J+4}. 
 \label{N^J+1_r}
\end{align}
\end{lemma}

\begin{proof}
We split the proof into the following two cases:
\[ \text{(i) }  \Pi_2(\TT_J) = \{r_2\} 
\qquad \text{and} \qquad \text{(ii) } 
 \Pi_2(\TT_J) \ne \{r_2\},\]

\noi
where $\Pi_2$ denotes the projection defined in \eqref{proj1}.

\smallskip

\noi
$\bullet$ {\bf Case (i):} 
We first consider  the case  $\Pi_2(\TT_J) = \{r_2\}$.
Recall that 
for general $J \in \N$, 
we need to control the rapidly growing
 cardinality
$c_J = |\mathfrak{BT}(J)| $  defined in \eqref{cj1}.
By Cauchy-Schwarz inequality and 
Lemma \ref{LEM:sum1}, we have
\begin{align*}
| \NN^{(J+1)}_0(v)|
& \lesssim  \|v\|_{L^2}
\sum_{\substack{\TT_J \in \mathfrak{BT}(J)\\\Pi_2(\TT_J) = \{r_2\}}}
\Bigg\{ \sum_{n \in \Z} 
\bigg(
 \sum_{\substack{{\bf n} \in \mathfrak{N}(\TT_J)\\{ n}_r = n}}
\ind_{\bigcap_{j = 1}^{J-1} C_j^c}
\frac{\jb{n}^{4s}}{|\phi_1|^{2}} 
\prod_{j = 2}^J \frac{1}{|\wt{\phi}_j|^2}
\bigg) \notag \\
& \hphantom{XXXXXX} \times 
\bigg(\sum_{\substack{{\bf n} \in \mathfrak{N}(\TT_J)\\{ n}_r = n}} 
\prod_{a \in \TT^\infty_J\setminus\{r_2\}} |v_{n_a}|^2
\bigg)\Bigg\}^\frac{1}{2}  \notag \\
& \lesssim \frac{c_J}{\prod_{j = 2}^J(2j+2)^{\frac{3}{2}-}}
 \|v\|_{L^2}^{2J+2}.
\end{align*}

\smallskip

\noi
$\bullet$ {\bf Case (ii):} 
Next, we  consider the case  $\Pi_2(\TT_J) \ne \{r_2\}$.
In this case, we need to modify the argument above since
the frequency $n_r = n$ does not correspond to a terminal node.
Note that 
$\TT_J^\infty = \Pi_1(\TT_J)^\infty \cup \Pi_2(\TT_J)^\infty$
and 
\begin{align}
\sum_{\substack{{\bf n} \in \mathfrak{N}(\TT_J)\\{  n}_r = n}} 
\prod_{a \in \TT^\infty_J} |v_{n_a}|^2
= \prod_{j = 1}^2 \bigg(\sum_{\substack{{\bf n} \in \mathfrak{N}(\Pi_j(\TT_J))\\{  n}_{r_j} = n}} 
\prod_{a_j \in \Pi_j(\TT_J)^\infty} |v_{n_{a_j}}|^2
\bigg).
\label{TT1}
\end{align}

\noi
Then, 
by Cauchy-Schwarz inequality and Lemma \ref{LEM:sum1}
with \eqref{TT1}, 
we have
\begin{align*}
| \NN^{(J+1)}_0(v)|
& \lesssim 
\sum_{\substack{\TT_J \in \mathfrak{BT}(J)\\\Pi_2(\TT_J) = \{r_2\}}}
\sum_{n \in \Z} 
\Bigg\{ 
\bigg(
 \sum_{\substack{{\bf n} \in \mathfrak{N}(\TT_J)\\{ n}_r = n}}
\ind_{\bigcap_{j = 1}^{J-1} C_j^c}
\frac{\jb{n}^{4s}}{|\phi_1|^{2}} 
\prod_{j = 2}^J \frac{1}{|\wt{\phi}_j|^2}
\bigg) \notag \\
& \hphantom{XXXXXX} \times 
\bigg(\sum_{\substack{{\bf n} \in \mathfrak{N}(\TT_J)\\{n}_r = n}} 
\prod_{a \in \TT^\infty_J} |v_{n_a}|^2
\bigg)\Bigg\}^\frac{1}{2}  \notag \\
& \les 
\frac{c_J}{\prod_{j = 2}^J(2j+2)^{\frac{3}{2}-}}
\sup_{\substack{\TT_J \in \mathfrak{BT}(2)\\ \Pi_2(\TT_J) \ne \{r_2\}}}
\sum_{n \in \Z}  \Bigg\{ \prod_{j = 1}^2 
\bigg( \sum_{\substack{{\bf n} \in \mathfrak{N}(\Pi_j(\TT_J))\\{  n}_{r_j} = n}} 
\prod_{a_j \in \Pi_j(\TT_J)^\infty} |v_{n_{a_j}}|^2\bigg)\Bigg\}^\frac{1}{2}\notag 
\intertext{By Cauchy-Schwarz inequality in $n$, }
& \lesssim \frac{c_J}{\prod_{j = 2}^J(2j+2)^{\frac{3}{2}-}}
\sup_{\substack{\TT_J \in \mathfrak{BT}(2)\\ \Pi_2(\TT_J) \ne \{r_2\}}}
   \prod_{j = 1}^2 \Bigg\{
\bigg( \sum_n \sum_{\substack{{\bf n} \in \mathfrak{N}(\Pi_j(\TT_J))\\{  n}_{r_j} = n}} 
\prod_{a_j \in \Pi_j(\TT_J)^\infty} |v_{n_{a_j}}|^2\bigg)\Bigg\}^\frac{1}{2}\notag  \\
& \lesssim \frac{c_J}{\prod_{j = 2}^J(2j+2)^{\frac{3}{2}-}}
 \|v\|_{L^2}^{2J+2}.
\end{align*}

\noi
This proves the first estimate \eqref{N^J+1_0-1}.

The second estimate \eqref{N^J+1_r}
 follows from  \eqref{N^J+1_0-1}  and $\l^2_n\subset\l^6_n$,
noting that, given $\TT_J \in \mathfrak{BT}(J)$, we have $\#\{b: b \in \TT^\infty_J\} = 2J+2$.
\end{proof}

Next, we treat $\NN^{(J+1)}$ in \eqref{N^J+1}.
As before, we write
\begin{equation} \label{N^J+1_ZZ}
\NN^{(J+1)} = \NN^{(J+1)}_1 + \NN^{(J+1)}_2,
\end{equation}

\noi
where $\NN^{(J+1)}_1$ is the restriction of $\NN^{(J+1)}$
onto $C_J$ defined in \eqref{Cj}
and
$\NN^{(J+1)}_2 := \NN^{(J+1)} - \NN^{(J+1)}_1$.
In the following lemma, we estimate the first term 
in \eqref{N^J+1_1}:
\begin{align} 
\NN^{(J+1)}_1 (v)
= -  \sum_{\TT_{J+1} \in \mathfrak{BT}(J+1)}
\sum_{{\bf n} \in \mathfrak{N}(\TT_{J+1})}
\ind_{(\bigcap_{j = 1}^{J-1} C_j^c) \cap C_J}
\frac{\jb{n_r}^{2s}e^{- i \wt{\phi}_{J+1}t } }{\prod_{j = 1}^J \wt{\phi}_j}
\,\prod_{a \in \TT^\infty_{J+1}} v_{n_{a}}.
\label{N^J+1_X}
\end{align}

\noi
Then,  we apply a normal form reduction 
once again to the second term $\NN^{(J+1)}_2$ 
as in \eqref{N^J+1}.
In Subsection \ref{SUBSEC:error}, 
we show that the error term 
$\NN^{(J+1)}_2$ tends to $0$ as $J \to \infty$.

\begin{lemma}\label{LEM:N^J+1_1}
Let $\NN^{(J+1)}_1$ be as in \eqref{N^J+1_X}.
Then, for any $s< 1$, 
we have
\begin{align} 
| \NN^{(J+1)}_1(v)| 
\lesssim 
 \frac{ J^{\frac{3}{2}+}\cdot c_{J+1}}{\prod_{j = 2}^J(2j+2)^{\frac{3}{2}-}}
\|v\|_{L^2}^{2J+4}.
\label{N^J+1_1}
\end{align}
\end{lemma}

\begin{proof}
We only discuss the case  $\Pi_2(\TT_{J+1}) = \{r_2\}$
since the modification for the case  $\Pi_2(\TT_{J+1}) \ne \{r_2\}$ is straightforward as in 
the proof of Lemma \ref{LEM:N^J+1_0}. 
By Cauchy-Schwarz inequality and 
Lemma \ref{sum1}, we have
\begin{align*}
| \NN^{(J+1)}_1(v)|
& \lesssim  \|v\|_{L^2}
\sum_{\substack{\TT_{J+1} \in \mathfrak{BT}(J+1)\\\Pi_2(\TT_{J+1}) = \{r_2\}}}
\Bigg\{ \sum_{n \in \Z} 
\bigg(
 \sum_{\substack{{\bf n} \in \mathfrak{N}(\TT_{J+1})\\{ n}_r = n}}
\ind_{(\bigcap_{j = 1}^{J-1} C_j^c) \cap C_J}
\frac{\jb{n}^{4s}}{|\phi_1|^{2}} 
\prod_{j = 2}^J \frac{1}{|\wt{\phi}_j|^2}
\bigg) \notag \\
& \hphantom{XXXXXX} \times 
\bigg(\sum_{\substack{{\bf n} \in \mathfrak{N}(\TT_{J+1})\\{ n}_r = n}} 
\prod_{a \in \TT^\infty_{J+1}\setminus\{r_2\}} |v_{n_a}|^2
\bigg)\Bigg\}^\frac{1}{2}  \notag \\
& \les  \frac{ J^{\frac{3}{2}+}\cdot c_{J+1}}{\prod_{j = 2}^J(2j+2)^{\frac{3}{2}-}}
\|v\|_{L^2}^{2J+4}.
\end{align*}

\noi
This proves \eqref{N^J+1_1}.
\end{proof}

\begin{remark} \rm

A  notable difference from \cite{GKO} appears in our definition of $C_j$
in \eqref{Cj};
on the one hand,   the cutoff size on $\wt \phi_{j+1}$ in \cite{GKO}
depended on $\wt \phi_j$ and $\phi_1$.
On the other hand, our choice of the cutoff size on $\wt \phi_{j+1}$ in \eqref{Cj}
is independent of  $\wt \phi_j$ or $\phi_1$,
thus providing simplification of the argument.

Another difference appears in the first step of the normal form reductions.
On the one hand, 
we simply applied the first normal form reduction in  \eqref{N^1}
without introducing a cutoff on the phase function $\phi_1$.
On the other hand, 
 in \cite{GKO}, a cutoff of the form\footnote{In \cite{GKO}, 
 this parameter was denoted by $N$.  Here, we use $K$ to avoid 
 the confusion with the frequency truncation parameter $N \in \N$.} $|\phi_1| >K$ was introduced
 to separate the first multilinear term
 into the nearly resonant and non-resonant parts.
The use of this extra parameter $K = K(\|u_0\|_{L^2})$
allowed the authors to show that the local existence time
can be given by $T \sim \|u_0\|_{L^2}^{-\al}$ for some $\al > 0$.
See \cite{GKO} for details.
Since our argument only requires the summability 
(in $J$) of the multilinear forms, 
we do not need to introduce this extra parameter.

\end{remark}

We conclude this subsection by briefly discussing how to justify all the formal steps
performed in the normal form reductions.
In particular, we need to justify 
\begin{itemize}

\item[(i)] the application of the product rule and 

\item[(ii)] switching time derivatives and summations

\end{itemize}

\noi
Suppose that a solution $v$ to \eqref{NLS5} lies
in $C(\R; H^\frac{1}{6}(\T))$.
Then,  
from \eqref{NLS5}, we have
\begin{equation*} 
\|\dt v_n\|_{C_T \ell^\infty_n} 
\leq \big\|  \F^{-1}(|\ft v_n|)\big\|_{C_T L^3_x}^3 
\les \big\|  \F^{-1}(|\ft v_n|)\big\|_{C_T H^\frac{1}{6}_x}^3  
=  \| v\|_{C_T H^\frac{1}{6}_x}^3  
\end{equation*}

\noi
for each $T > 0$, where $C_TB_x = C([-T, T]; B_x)$.
Hence, $\dt v_n \in C([-T, T]; \ell^\infty_n)$,
justifying (i) the application of the product rule.
Note that given $N \in \N$, any solution $v$ to \eqref{NLS7} belongs
to $C(\R; H^\infty(\T))$
and hence (i) is justified.
Moreover, the summations over spatial frequencies in the normal form reductions 
applied to solutions to \eqref{NLS7} 
are all finite
and therefore, (ii) the  switching time derivatives and summations 
over spatial frequencies trivially hold true
for \eqref{NLS7}.
In general, 
the proof of Lemma \ref{LEM:N^J+1_0}
shows that the summation defining  $\NN_0^{(j)}$
converges  (absolutely and uniformly in time). 
Then, Lemma 5.1 in~\cite{GKO} allows us to switch 
the time derivative with the summations
as temporal distributions, thus
justifying differentiation by parts.

\subsection{On the error term $\NN_2^{(J+1)}$}
\label{SUBSEC:error}

In this subsection, we prove that $\NN_2^{(J+1)}$
in \eqref{N^J+1_ZZ}
tends to 0 as $J \to \infty$ under some regularity assumption on $v$.
From \eqref{N^J+1}, we have
\begin{align} \label{X1}
 \NN_2^{(J+1)}(v)
&  =- \sum_{\TT_{J+1} \in \mathfrak{BT}(J+1)}
\sum_{{\bf n} \in \mathfrak{N}(\TT_{J+1})}
\ind_{\bigcap_{j = 1}^{J} C_j^c}
\frac{\jb{n_r}^{2s}e^{- i \wt{\phi}_{J+1}t } }{\prod_{j = 1}^J \wt{\phi}_j}
\,\prod_{a \in \TT^\infty_{J+1}} v_{n_{a}}.
\end{align}

\begin{lemma}\label{LEM:N^J+1_2}
Let $\NN^{(J+1)}_2$ be as in \eqref{X1}.
Then, given $v \in H^s(\T)$, $s > \frac 12$,  
we have
\begin{align} 
| \NN^{(J+1)}_2(v)|  \longrightarrow 0, 
\label{X1a}
\end{align}

\noi
as $J \to \infty$.
\end{lemma}

We point out that one can actually prove Lemma \ref{LEM:N^J+1_2}
under a weaker regularity assumption $s \geq \frac 16$.
See \cite{OW}.
For our purpose, however, we only need to prove the vanishing of the error term 
$\NN^{(J+1)}_2$ for sufficiently regular functions;
our main objective is to obtain an energy estimate 
(on the modified energy $\EE_{N, t}$ defined in \eqref{energy6}) 
for solutions to the truncated equation \eqref{NLS7}.
Given $N \in \N$, any solution $v$ to \eqref{NLS7} belongs
to $C(\R; H^\infty(\T))$.
Therefore, while 
 the convergence speed in \eqref{X1a}  depends on $N \in \N$, 
 the final energy estimate
\eqref{energy7} holds with an implicit constant  {\it independent} of $N \in \N$.

\begin{proof}

Given ${\bf n} \in \mathfrak{N}(\TT_{J+1})$, 
it follows from Definition \ref{DEF:tree4}  and the triangle inequality that 
there exists $C_0 > 0$ such that 
\begin{align}
|n_r| \leq C_0^J |n_{b_k}|
\label{X2}
\end{align}

\noi
for (at least) 
 two terminal nodes $b_1, b_2 \in \TT_{J+1}^\infty$
Then, by Young's inequality (placing $v_{n_{b_k}}$ in $\l^2_n$, $k = 1, 2$, 
and the rest in $\l^1_n$) with  \eqref{cj1}, \eqref{Cj}, and \eqref{X2}, we have 
\begin{align*} 
 |\NN_2^{(J+1)}(v)|
& \les \frac{C_0^{2sJ}\cdot c_J}{\prod_{j = 2}^J(2j+2)^3}
  \sup_{\TT_{J+1} \in \mathfrak{BT}(J+1)}
\sum_{{\bf n} \in \mathfrak{N}(\TT_{J+1})}
\bigg(\prod_{k = 1}^2 \jb{n_{b_k}}^{s}|v_{n_{b_k}}|\bigg)
\bigg(\prod_{a \in \TT^\infty_{J+1}\setminus{\{b_1, b_2\}}} |v_{n_{a}}|\bigg)\\
& \les \frac{C_0^{2sJ}\cdot c_J}{\prod_{j = 2}^J(2j+2)^3}
\|v\|_{H^s}^{2J+4} \longrightarrow 0, 
\end{align*}

\noi
as $J \to \infty$.
\end{proof}

\subsection{Improved energy bound}
\label{SUBSEC:improved}

We are now ready to establish
the improved energy estimate \eqref{energy7}.
Let $v$ be a smooth global solution\footnote{In fact, it suffices to assume that $v \in C(\R; H^\frac{1}{6}(\T))$.
See \cite{GKO, OW}.} to \eqref{NLS5}.
Then, by applying the normal form reduction $J$ times, we obtain\footnote{Once again, we are replacing $\pm 1$
and  $\pm i$
by 1 for simplicity since they play no role in our analysis.} 
\begin{align*} 
\dt \bigg(\frac 12 \| v \|_{H^s}^2\bigg)
=  \dt \sum_{j = 2}^{J+1} \NN^{(j)}_0(v)
 + \sum_{j = 2}^{J+1} \NN^{(j)}_1(v) + \sum_{j = 2}^{J+1} \RR^{(j)}(v)
 + \NN^{(J+1)}_2.
\end{align*}

\noi
Thanks to Lemma \ref{LEM:N^J+1_2}, 
by taking the limit as $J \to \infty$, we obtain
\begin{align*} 
\dt \bigg(\frac 12 \| v \|_{H^s}^2\bigg)
=  \dt \sum_{j = 2}^\infty \NN^{(j)}_0(v)
 + \sum_{j = 2}^\infty \NN^{(j)}_1(v) + \sum_{j = 2}^\infty \RR^{(j)}(v).
\end{align*}

\noi
In other words,  defining the modified energy $\EE_t(v)$
by 
\begin{align*}
\EE_t(v) := \frac 12 \| v (t) \|_{H^s}^2- \sum_{j = 2}^\infty \NN^{(j)}_0(v)(t), 
\end{align*}

\noi
we have
\begin{align*} 
\dt \EE_t(v) 
=  \sum_{j = 2}^\infty \NN^{(j)}_1(v)(t) + \sum_{j = 2}^\infty \RR^{(j)}(v)(t).
\end{align*}

\noi
Suppose that  $\|v\|_{C(\R; L^2)} \leq r$.
Then, applying Lemmas  \ref{LEM:N^J+1_0} and \ref{LEM:N^J+1_1} with \eqref{cj1}, 
we obtain 
\begin{align*} 
|\dt \EE_t(v) |
& \les
\sum_{j =2}^\infty
 \frac{c_{j-1}}{\prod_{k = 2}^{j-1}(2k+2)^{\frac{3}{2}-}}
r^{2j}
+ \sum_{j =2}^\infty
 \frac{ j^{\frac{3}{2}+}\cdot c_{j}}{\prod_{k = 2}^{j-1}(2k+2)^{\frac{3}{2}-}}
r^{2j+2}\\
& + \sum_{j =2}^\infty
 \frac{ j\cdot c_{j-1}}{\prod_{k = 2}^{j-1}(2k+2)^{\frac{3}{2}-}}
r^{2j+2}\\
& \leq C(r).
\end{align*}

\noi
In view of the boundedness of the
frequency projections
and noting that any solution to~\eqref{NLS7} is in $H^\infty(\T)$, 
the same energy estimate holds for solutions to the truncated equation~\eqref{NLS7},
uniformly in  $N \in \N$.

\subsection{On the proof of Lemma \ref{LEM:Gibbs}}
\label{SUBSEC:finite}

We conclude this section with a brief discussion on the proof of  Lemma \ref{LEM:Gibbs}.
First,  note that Lemma~\ref{LEM:Gibbs}\,(ii) is an immediate corollary
of Lemma~\ref{LEM:Gibbs}\,(i).
Moreover,  Lemma~\ref{LEM:Gibbs}\,(i) follows from Egoroff's theorem 
once we prove that 
\[\mathfrak{S}_N( v) = \sum_{j = 2}^\infty \NN^{(j)}_{0, N}( v)\]

\noi
converges almost surely to 
\[\mathfrak{S}_\infty(v) = \sum_{j = 2}^\infty \NN^{(j)}_0( v).\]

\noi
See   \cite[Proposition 6.2]{OTz1}.
In fact, one can show that 
$\mathfrak{S}_N( v)$ converges to $\mathfrak{S}_\infty( v)$
for any $v \in L^2(\T)$.

Recall from \eqref{N^J+1} that 
 $\NN^{(j)}_0( v)$ consists of a sum of the multilinear forms
associated with ordered bi-trees $\TT_{j-1} \in \mathfrak{BT}(j-1)$.
Given $N \in \N$, the multilinear form 
$\NN^{(j)}_{0, N}( v)$
is obtained 
in a similar manner
with the following modifications:

\begin{itemize}
\item[(i)] We set 
$ v_n = 0$ for all $|n| >N$. 
This corresponds to setting 
$v_{n_a} = 0$  for all $|n| >N$ and all terminal nodes $a \in \TT_{j-1}^\infty$.

\item[(ii)] In view of \eqref{NLS7},
we also set 
$v_{n_a} = 0$  for all $|n| >N$ and all parental nodes in $\TT_{j-1}$.
This amounts to setting 
$v_{n_a} = 0$  for all $|n| >N$ and all non-terminal nodes $a \in \TT_{j-1}^0$.

\end{itemize}

\noi
In particular, we have
\begin{align}
\NN^{(j)}_{0, N} ( v)
= \sum_{\TT_{j-1} \in \mathfrak{BT}(j-1)}
 \sum_{{\bf n} \in \mathfrak{N}_N(\TT_{j-1})}
\ind_{\bigcap_{k = 1}^{j-2} C_k^c}
\frac{\jb{n_r}^{2s}e^{- i \wt{\phi}_{j-1}t } }{\prod_{k = 1}^{j-1} \wt{\phi}_k}
\, \prod_{a \in \TT^\infty_{j-1}} v_{n_{a}}, 
\label{NFX}
\end{align}

\noi
where $ \mathfrak{N}_N(\TT_{j-1})$ is as in Definition \ref{DEF:tree4}.
Namely, 
$\NN^{(j)}_{0, N} (v)$ is obtained from $\NN^{(j)}_{0} (v)$
by simply truncating 
all the frequencies (including the ``hidden''\footnote{Namely, the parental frequencies
at the non-terminal nodes do not appear explicitly in the sum in~\eqref{NFX}
but they implicitly appear through the relation in Definition \ref{DEF:tree4}.} parental frequencies)
by $N \in \N$.

Write 
 \begin{align}
 \NN^{(j)}_0( v) - \NN^{(j)}_{0, N}(  v)
&  = \Big\{\NN^{(j)}_0( v) - \wt \NN^{(j)}_{0, N}(  v)\Big\}
 + \Big\{\wt \NN^{(j)}_{0, N}( v) - \NN^{(j)}_{0, N}(  v)\Big\} \notag \\
& =: \1_j + \II_j, 
\label{NFX1}
 \end{align}

 \noi where $\wt \NN^{(j)}_{0, N} ( v)$ is 
 obtained from 
  $\NN^{(j)}_{0, N} ( v)$ 
  by replacing 
$ {\bf n} \in \mathfrak{N}_N(\TT_{j-1})$
with 
$ {\bf n} \in \mathfrak{N}_N^0(\TT_{j-1})$. 
i.e.\,we are truncating only the parental frequencies
at non-terminal nodes $a \in \TT^0_{j-1}$.
 Then, by writing
\[  \II_j = \wt \NN^{(j)}_{0, N}( v) - \wt \NN^{(j)}_{0, N}(  \P_{\leq N} v), \]

\noi
 it follows from the multilinearity and the boundedness in $L^2(\T)$ (Lemma \ref{LEM:N^J+1_0})
that the second term $\II_j$  tends to 0 as $N \to \infty$, 
by simply writing the difference in a telescoping sum.
More precisely, we write  $\II$ as a telescoping sum, replacing
$2j$ factors of $v_{n_a}$, $a \in \TT^\infty_{j-1}$, into 
$2j$ factors of $(\P_{\leq N} v)_{n_a}$.
This 
introduces
$2j$ differences, each containing exactly one
factor of $v - \P_{\leq N} v$
(tending to 0 as $N \to \infty$). 
We then simply apply Lemma~\ref{LEM:N^J+1_0} on each difference.

Similarly, we can show that  $\1_j$ in \eqref{NFX1} tends to 0 as $N \to \infty$ 
by writing the difference in a telescoping sum.
Namely, noting only the difference
between $\NN^{(j)}_0( v)$ and $ \wt \NN^{(j)}_{0, N}(  v)$
is the frequency cutoffs $\ind_{|n_a| \leq N}$ at each non-terminal node $a \in \TT^0_{j-1}$,
we introduce $j-1$ differences by adding the frequency cutoff $\ind_{|n_a| \leq N}$ at each non-terminal node
in a sequential manner.
By construction, each of the $j-1$ differences
has one non-terminal node $a_* \in \TT^0_{j-1}$ with the restriction $|n_{a_*}| > N$.
Then, from Definition \ref{DEF:tree4}, 
we see that there exists at least one terminal node $b \in  \TT^\infty_{j-1}$
which is a descendants of $a_*$ such that 
\begin{align*}
|n_b| \geq C_0^{-J} |n_{a_*}| > C_0^{-J}N 
\end{align*}

\noi
for some  $C_0 > 0$ (compare this with \eqref{X2}).
This forces each of the $j-1$ differences in the telescoping sum
to tend to 0 as $N \to \infty$, 
and hence $\1_j$ in \eqref{NFX1} tends to 0 as $N \to 0$.

Therefore, 
$\NN^{(j)}_{0, N}(  v)$ converges to 
$ \NN^{(j)}_0( v)$ as $n \to \infty$ for any $v \in L^2(\T)$.
Finally, in view of the fast decay in $j$ in Lemma \ref{LEM:N^J+1_0}, 
the convergence of $\mathfrak{S}_N( v)$  to $\mathfrak{S}_\infty( v)$
follows from 
the dominated convergence theorem.

\section{Proof of Theorem \ref{THM:sing}:
Non quasi-invariance under the dispersionless model}
\label{SEC:sing}

In this section, we present the proof of 
Theorem \ref{THM:sing}.  The basic ingredients are
the Fourier series representation of the (fractional)
Brownian loops,  the  law of the iterated logarithm,
and the solution formula \eqref{formula} to the dispersionless model \eqref{ND1}.
More precisely, 
we show that, while the Gaussian random initial data distributed according to $\mu_s$
satisfies the law of the iterated logarithm, 
the solution given by \eqref{formula} does not satisfy
the law of the iterated logarithm for any non-zero time.
We divide the argument into three cases:
(i) $s = 1$ corresponding to 
the Brownian/Ornstein-Uhlenbeck loop,
(ii) $\frac 12 < s < \frac 32$, 
corresponding to the fractional Brownian loop
(and $s > \frac 32$ with $s \notin \frac 12 + \N$),
and (iii) $s \in \frac 12 + \N$: the critical case.
For simplicity, we set $t = 1$ in the following.  The proof for non-zero $t \ne1$ follows in a similar manner.

\subsection{Brownian/Ornstein-Uhlenbeck loop}\label{SUBSEC:5.1}
We first consider the $s= 1$ case.
Under the law of the random Fourier series
\begin{equation}\label{BL}
u(x) =u(x; \o) = \sum_{n\in \mathbb{Z}}\frac{g_n(\o)}{\jb{n}}e^{inx}
\end{equation}

\noi
corresponding to the Gaussian measure $\mu_1$, 
$\Re u$ and $\Im u$ are independent stationary Ornstein-Uhlenbeck (OU) processes (in $x$)
on $ [0,2\pi)$. Recall that the law of this process can be written as
\begin{equation}\label{rep1}
u(x) \stackrel{\mathrm{d}}{=} \P_{\ne 0}w(x) + g_0 
= w(x) -\fint_0^{2\pi}w(y) dy  + g_0,
\end{equation}

\noi
where $w$ is a complex OU bridge with $w(0)=w(2\pi)=0$ 
and $g_0$ is a standard complex-valued Gaussian random variable
(independent from $w$).

We now recall the law of the iterated logarithm for the Brownian motion (see \cite[I.16.1]{RW}): 
\begin{proposition}\label{LEM:LIL1}
Let $B(t)$ be  a standard Brownian motion on $\R_+$. Then, for each $t \geq 0$,
\begin{equation}\label{LIL1}
\limsup_{h\downarrow 0} \frac{B(t+h)-B(t)}{\sqrt{2h\log\log \frac{1}{h}}}=1,
\end{equation}
almost surely.
\end{proposition}

It follows from 
 the representation \eqref{rep1},
 the absolute continuity\footnote{The absolute continuity property
 claimed here can be easily seen by the Fourier series representations
  of the Brownian motion/bridge (with \eqref{BL} and \eqref{rep1})
  and Kakutani's theorem (Lemma \ref{LEM:Kakutani} below).
  For example, the Brownian motion $B(t)$ on $[0, 2\pi)$ has
  the following Fourier-Wiener series
  \[B(t) = g_0 t + \sum_{n \in \Z \setminus \{0\}} \frac{g_n}{n} e^{int}.\]   }
  of the Brownian bridge with respect to the Brownian motion
 on any interval $[0,\gamma)$,  $\gamma<2\pi$, 
 and the absolute continuity of the OU bridge with respect to the Brownian bridge also on any interval $[0,\gamma)$, 
 that the limit \eqref{LIL1} also holds for $\Re u$ and $\Im u$ on $[0, 2 \pi)$. 

Define $\psi$ by 
\[\psi(h) = \sqrt{2h\log\log \frac{1}{h}}, \quad 0<h<1.\]

\noi
Let $0 \leq  x < 2\pi$.
As a corollary to Proposition \ref{LEM:LIL1}, we have 
\begin{align}
\label{LIL2}
& \limsup_{h\downarrow 0} \frac{\Re u(x+h)-\Re u(x)}{\psi(h)}=1
\end{align}

\noi
almost surely.

In the following, 
by a direct calculation, 
we show that 
$\Re[e^{-i  |u|^2}u]$ does not satisfy \eqref{LIL2} with a positive probability.
This will show that 
the pushforward measure $\wt \Phi(t)_* \mu_s$
under the dynamics of the dispersionless model \eqref{ND1}
is not absolutely continuous with respect to the Gaussian measure $\mu_s$. 

On the one hand, we have 
\begin{align*}
 \Re [& e^{-i|u(y)  |^2} u(y)] -\Re[e^{-i|u(x)|^2}u(x)]\notag \\
& =(\Re u(y)-\Re u(x))\cos |u(y)|^2+(\cos|u(y)|^2-\cos|u(x)|^2)\Re u(x) \notag\\
& \hphantom{X}+(\Im u(y)-\Im u(x))\sin |u(y)|^2+(\sin |u(y)|^2-\sin|u(x)|^2)\Im u(x). 
\end{align*}

\noi
On the other hand, 
by the Taylor expansion
with $\eta(x, y) = |u(y)|^2-|u(x)|^2$, we have
\begin{align*}
\cos |u(y)|^2 &=\cos |u(x)|^2 -\sin |u(x)|^2 \cdot \eta(x, y) + O\big(\eta^2(x, y)\big), \\
\sin |u(y)|^2 &= \sin |u(x)|^2 +\cos |u(x)|^2 \cdot \eta(x, y)+O\big(\eta^2(x, y)\big).
\end{align*}

\noi
Putting together, we obtain
\begin{align}
\Re [e^{- i|u(y)|^2}& u(y)] -\Re[e^{- i|u(x)|^2}u(x)] \notag\\
& =  (\Re u(y)-\Re u(x))\cos |u(y)|^2
 - \sin |u(x)|^2 \cdot \eta(x, y) \Re u(x) \notag \\
& \hphantom{X}
 + (\Im u(y)-\Im u(x))\sin |u(y)|^2
 +\cos|u(x)|^2\cdot \eta(x, y) \Im u(x)\notag \\
& \hphantom{X}
  + O\big(\eta^2(x, y)\big)\cdot(|\Re u(x)| +|\Im u(x)|)\notag \\
& =  (\Re u(y)-\Re u(x))\cos |u(y)|^2\notag \\
& \hphantom{X}
- \sin |u(x)|^2 \cdot\big\{(\Re u(y) - \Re u(x))(\Re u(y)+\Re u(x)) \big\} \Re u(x)\notag  \\
& \hphantom{X}
-\sin |u(x)|^2 \cdot\big\{(\Im u(y) - \Im u(x))(\Im u(y)+\Im u(x)) \big\} \Re u(x) \notag \\
& \hphantom{X}
 +(\Im u(y)-\Im u(x))\sin |u(y)|^2\notag \\
& \hphantom{X}
+\cos|u(x)|^2\cdot\big\{(\Re u(y)-\Re u(x))(\Re u(y)+\Re u(x))\big\} \Im u(x)\notag \\
& \hphantom{X}
+\cos|u(x)|^2\cdot\big\{(\Im u(y)-\Im u(x))(\Im u(y)+\Im u(x))\big\} \Im u(x)\notag \\
& \hphantom{X}
  + O\big(\eta^2(x, y)\big)\cdot(|\Re u(x)| +|\Im u(x)|).
\label{LIL4}
\end{align}

Fix $0 \leq  x < 2\pi$.
Let $\{h_n = h_n(\o)\}_{n \in \N}$ be a (random) sequence achieving the limit supremum in \eqref{LIL2}
almost surely.
Then, for this sequence $\{h_n\}_{n \in \N}$, we have
\begin{align}
& \limsup_{n \to \infty} \frac{|\Im u(x+h_n)-\Im u(x)|}{\psi(h_n)}\leq  1
\label{LIL3}
\end{align}

\noi
almost surely.
Divide the expression in \eqref{LIL4} by $\psi(h_n)$,
after replacing $y$ by $x+h_n$.
Then, by  taking the  limit as $n \to \infty$
and applying \eqref{LIL2} and \eqref{LIL3}, we have
\begin{align}
  \limsup_{n\rightarrow \infty} & 
\frac{\Re [e^{- i|u(x+h_n)|^2}u(x+h_n)]-\Re[e^{-i|u(x)|^2}u(x)]}{\psi(h_n)}
\notag\\
&\ge - 2\sin |u(x)|^2 \cdot (\Re u(x))^2\notag\\
&\phantom{X} - \big|\cos|u(x)|^2\big| -\big|\sin |u(x)|^2\big|
- 2\cdot \big|\sin |u(x)|^2 \cdot \Im u(x)\Re u(x)\big|\notag\\
&\phantom{X}
 - 2\cdot \big|\cos |u(x)|^2\cdot  \Re u(x) \Im u(x)\big|
- 2\cdot\big| \cos |u(x)|^2\big| (\Im u(x))^2,
\label{LIL5}
\end{align}

\noi
almost surely.

Fix $M \gg1 $ by 
\begin{align}
M^2= - \frac{\pi}{2}+2k\pi,
\label{LIL6}
\end{align}

\noi
for some large $k \in \N$ (to be chosen later).
Given $\eps>0$, define the set
\[A =\big\{\o \in \O: |\Re u(x;\o )-M|\le \eps , \  |\Im u(x;\o)|\le \eps \big\}.\]

\noi
Noting that under the law of the OU loop, $\Re u(x)$ and $\Im u(x)$ are independent Gaussian
random variables, we have
\begin{equation*}
P(A) \ge \delta(M,\epsilon)>0
\end{equation*}

\noi
for any $\eps>0$. 
By choosing $\eps> 0$ sufficiently small such that $\eps M \ll 1$, 
we have 
\begin{align}
 \big||u(x)|^2-M^2\big| \leq 2 \eps (M+\eps) = o(1).
\label{LIL7}
\end{align}

\noi
Then, from \eqref{LIL6} and \eqref{LIL7}, we have
\begin{align*}
\text{RHS of }\eqref{LIL5}
& \geq 2\big|\sin|u(x)|^2  \big| \cdot M (M-3\eps)- 2 (1 + \eps (M + 2\eps))\\
& \geq M^2 
\end{align*}

\noi
on $A$.
By choosing $M \gg1$, we see that the set
\[A_1 = \left\{\limsup_{h\downarrow 0} \frac{\Re[e^{- i|u(x+h)|^2}u(x+h)]-\Re[e^{-i|u(x)|^2]}u(x)]}{\psi(h)}=1\right\}\]

\noi
does not have probability 1 under the law of $u$.
Therefore, $\mu_1$ is not quasi-invariant
under the flow of the dispersionless model \eqref{ND1}.

\subsection{Fractional Brownian motion}\label{SUBSEC:5.2}

In this subsection, we
extend the previous result to  the distribution of the random Fourier series
\begin{align}
u_s(x) & = \sum_{n\in \mathbb{Z}} \frac{g_n}{\jb{n}^s}e^{inx}, \label{fBM}
\end{align}

\noi
corresponding to $\mu_s$.
For $\frac{1}{2}<s<\frac{3}{2}$, the series \eqref{fBM} is related to 
a {\it fractional Brownian motion}. 
Recall that a fractional Brownian motion with Hurst parameter $H$, $0<H<1$, 
is the Gaussian process $B^H(t),t\ge 0 $ with stationary increments and covariance
\begin{equation}\label{COV}
\mathbb{E}[B^H(t_1)B^H(t_2)] = \frac{\rho(H)}{2}(t_1^{2H}+t_2^{2H}-|t_2-t_1|^{2H}),
\end{equation}

\noi
where
\[\rho(H) = \mathbb{E}\big[\big(B^H(1)\big)^2\big] = -2 \frac{\cos (\pi H)}{\pi}\G(-2H) \ \text{ when } H \ne \tfrac12
\qquad \text{and}\qquad  \rho\big(\tfrac 12\big)=1.\]

\noi
When  $H=\frac 12$, a fractional Brownian motion becomes the standard Brownian motion.
In the following, we only consider the case $H \ne \frac 12$.

It is known that there is a subtle issue on building a series representation for
a fractional Brownian motion $B^H$.
Instead, we consider the following series\footnote{As mentioned in Section \ref{SEC:1}, 
 we drop the factor  of $2\pi$.}
\[\ft {B}^H (t) = \wt g_0 t+\sqrt{2}\sum_{n\ge 1} \bigg(\wt g_n\frac{\cos(nt)-1}{n^{H+\frac12}}+\wt g_n'
 \frac{\sin(nt)}{n^{H+\frac 12}}\bigg),\]

\noi
for $t\in [0,2\pi]$, where $\wt{g}_n$ and $\wt{g}_n'$ are now independent {\it real-valued} standard Gaussians. 
Then, Picard \cite{picard}  showed the following result
on the relation between $B^H$ and $\ft B^H$.

\begin{lemma}[Theorems 24 and 27 in Section 6 of \cite{picard}]\label{picardthm}
The processes $B^H(t)$ and $\ft {B}^H(t)$ can be coupled in such a way that 
\[B^H(t)-\widehat{B}^H(t)\]
is a $C^\infty$-function on $(0,2\pi]$.
Moreover, if $H \ne \frac 12$, then the laws of $B^H$ and $\ft B^H$
are equivalent on $[0, T]$ for $T < 2\pi$ (and mutually singular if $T = 2\pi$).
\end{lemma}

Recall Kakutani's criterion \cite{Kakutani} in the Gaussian case.

\begin{lemma}\label{LEM:Kakutani}
Let $\{g_n\}_{n\in \N}$ and $\{\wt{g}_n\}_{n\in \N}$ 
be two sequences of centered Gaussian random variables with variances 
$\mathbb{E}[g_n^2]=\sigma_n^2>0$ and $\mathbb{E}[\wt{g}_n^2] = \wt{\sigma}_n^2>0$.
Then,  the laws of the sequences $\{g_n\}_{n\in \N}$ and $\{\tilde{g}_n\}_{n\in \N}$ are equivalent if and only if
\[\sum_{n\in \N} \bigg( \frac{\wt{\s}^2_n}{\s^2_n}-1\bigg)^2<\infty.\]

\end{lemma}

From \eqref{fBM}, we have
\begin{align}
\Re u_s(x) 
& = \Re g_0 + \sum_{n \geq 1}
\bigg(  \frac{\Re g_n+\Re g_{-n}}{\jb{n}^s}\cos (n x)
+  \frac{-\Im g_n +\Im g_{-n}}{\jb{n}^s}\sin(nx)\bigg).
\label{fBM2}
\end{align}

\noi
Then, 
applying Lemma \ref{LEM:Kakutani}  to the sequences
\[\bigg\{\wt{g}_0 t, \ \sqrt{2}\frac{\wt g_n}{n^{H+\frac12}},\ \sqrt{2}\frac{\wt g_n'}{n^{H+\frac12}}\bigg\}\]
and
\[ \bigg\{  \Re g_0, \ \frac{\Re g_n+\Re g_{-n}}{\jb{n}^s}, \ \frac{-\Im g_n +\Im g_{-n}}{\jb{n}^s}\bigg\},\]

\noi
we see that if $s=H+\frac12$, then the series \eqref{fBM2}
for $\Re u_s$
and 
\begin{align}
\wt B^H :=  \ft {B}^H- \fint_0^{2\pi}(\ft {B}^H(\alpha)-\wt g_0\alpha)d \al
  \label{fBM2a}
\end{align}
have laws that are mutually absolutely continuous. 
Therefore, in view of the computation above
and Lemma \ref{picardthm} with \cite[Theorem 35]{picard}, 
we see that 
 the laws of $B^H$ and $\Re u_s$
are equivalent on $[0, T]$ for $T < 2\pi$.
The same holds for $\Im u_s$.

We use the following version of the law of the iterated logarithm for Gaussian processes with stationary increments \cite[Theorem 7.2.15]{marcusrosen}.
First, recall the following definition.
We say that 
a function $f$ is called a normalized regularly varying function at zero with index $\al >0$ if it can be written in
the form
\[ f(x) = Cx^\al \exp \bigg(\int_1^x \frac{\eps(u)}{u} du \bigg)\]

\noi
for some constant $C \ne 0$
and $\lim_{u \to 0} \eps(u) = 0$.

\begin{proposition}\label{thm:LIL}  
Let $G=\big\{G(x), x\in [0,2\pi]\big\}$ be a Gaussian process with stationary increments and let
\[\s^2(h) = \mathbb{E}\big[|G(h)-G(0)|^2\big].\]

\noi
If $\s^2(h)$ is a normalized regularly varying function at zero with index $0<\alpha<2$, then
\begin{equation*}
\lim_{\dl\rightarrow 0} \sup_{|h|\le \dl} \frac{|G(h)-G(0)|}{\sqrt{2\s^2(|h|)\log\log \frac1{|h|}}}=1
\end{equation*}

\noi
almost surely.
\end{proposition}

Using  the covariance \eqref{COV}, we have
\begin{align}
\s^2(h) = \E\big[|B^H(h)-B^H(0)|^2\big] = Ch^{2H}.
\label{sigma}
\end{align}

\noi
Hence, Proposition \ref{thm:LIL}
holds for  $G(x)=B^H(x)$, $H<1$. 
Then, by the absolute continuity, the conclusion of Proposition \ref{thm:LIL}
with 0 replaced by any $x \in (0, 2\pi)$ 
also holds for $\Re u_s$ and $\Im u_s$; for any $\frac 12<s<\frac32$, we have
\begin{align}
& \lim_{\dl\rightarrow 0} \sup_{|h|\le \dl} \frac{|\Re u_s(x+h)-\Re u_s(x)|}{\sqrt{2\s^2(|h|)\log\log \frac1{|h|}}}=1
\label{LIL11a}
\end{align}

\noi
almost surely. 
Applying the law of the iterated logarithm conditionally on the set where \eqref{LIL11a} holds, 
we also have 
\begin{align*}
& \lim_{\dl\rightarrow 0} \sup_{|h|\le \dl} \frac{|\Im u_s(x+h)-\Im u_s(x)|}{\sqrt{2\s^2(|h|)\log\log \frac1{|h|}}}\leq1
\end{align*}

\noi
almost surely.
We can now reproduce exactly the proof in the previous subsection for $\frac 12<s<\frac 32$.
This proves Theorem \ref{THM:sing} for $\frac 12<s<\frac 32$.

Next, we consider the case
 $s\geq \frac 32$ such that  $s\notin \frac 12+\N$.
 We consider the critical case 
 $s\in \frac 12+\N$ in the next subsection.
The main point is to note that  
   $u_s$ in \eqref{fBM} has a $C^r$-version for each integer 
   $r<\lfloor s-\tfrac 12\rfloor $. Indeed, we have the following:

\begin{lemma}\label{LEM:cov}
Let $X(t)$, $t\in \R$,  be a stationary Gaussian process with the covariance function
\[\rho(t) =\int e^{i\al t} \nu(d\al). \]

\noi
If $\int |\al|^{2+\eps}\nu(d\al)<\infty$ for some $\eps>0$, 
then there is a version of the process $X(t)$ such that 
$\dt X(t)$ exists and is continuous. 
Moreover, $\dt X(t)$ is a stationary Gaussian process with covariance
\begin{equation*}
\int e^{i\al t} \al^2 \nu(d \al).
\end{equation*}

\end{lemma}

Since we work on $\T$, the spectral measure $\nu(d\al)$ is 
the counting measure on $\Z$ and 
\[\rho(x) =2 \sum_{n\in \Z} \frac{e^{inx}}{\jb{n}^{2s}}.\]

\noi
Note that when $s>\frac 32$, we have
\[2\sum_{n\in \Z} \frac{|n|^{2+\eps}}{\jb{n}^{2s}}<\infty\]

\noi
for sufficiently small $\eps > 0$
and thus we can apply Lemma \ref{LEM:cov}.
Differentiating $u_s$ in \eqref{fBM} $\lfloor s-\frac 12 \rfloor$ times, 
we obtain a process 
\begin{align*}
 \dx^r u_s(x) \stackrel{\textrm{d}}{=}  
\sum_{n\in \Z \setminus\{0\}} \frac{ g_n}{|n|^{-r} \jb{n}^{s}}e^{inx}
\end{align*}

\noi
with $r = \lfloor s-\frac 12 \rfloor$. 
Given $ s \in (\frac 12 + j, \frac 32 + j)$ for some $j \in \N$, 
we have $s - r = s-j \in (\frac 12, \frac 32)$.
Noting that $|n|^{-r} \jb{n}^{s}\sim \jb{n}^{s-r}$, we can 
apply Lemma \ref{LEM:Kakutani} to show the laws of 
$\wt B^H$ defined in \eqref{fBM2a}, $\Re  \dx^r u_s$, 
and $\Im  \dx^r u_s$ are equivalent.
Hence, proceeding as before, we can apply  Proposition \ref{thm:LIL} to 
 $\Re  \dx^r u_s$ and $\Im  \dx^r u_s$.
 Namely, we obtain
 \begin{align}
& \lim_{\dl\rightarrow 0} \sup_{|h|\le \dl} \frac{|\Re \dx^r u_s(x+h)-\Re \dx^r u_s(x)|}{\sqrt{2\s^2(|h|)\log\log \frac1{|h|}}}=1
\label{LIL12}
\end{align}

\noi
almost surely. 
 Applying the law of the iterated logarithm conditionally on the set where \eqref{LIL12} holds, 
 we also have 
 \begin{align}
& \lim_{\dl\rightarrow 0} \sup_{|h|\le \dl} \frac{|\Im \dx^r u_s(x+h)-\Im \dx^r u_s(x)|}{\sqrt{2\s^2(|h|)\log\log \frac1{|h|}}}\leq 1
\label{LIL13}
\end{align}

\noi 
almost surely.

With \eqref{LIL12} and \eqref{LIL13} at hand, 
we  can basically repeat  the   proof 
in  Subsection~\ref{SUBSEC:5.1} 
by differentiating \eqref{LIL4}
and applying 
 \eqref{LIL12} and \eqref{LIL13}.
A straightforward application of the product rule to compute $\dx^r (e^{-i|u_s(x)|^2}u_s(x))$
would be computationally cumbersome.
Thus, we perform some simplification before taking derivatives.
First, note that from \eqref{LIL12} and \eqref{LIL13} with \eqref{sigma}, we have
 \begin{align}
& \lim_{\dl\rightarrow 0} \sup_{|h|\le \dl} 
\frac{|\Re \dx^j u_s(x+h)-\Re \dx^j u_s(x)|}{\sqrt{2\s^2(|h|)\log\log \frac1{|h|}}}=0,
\label{LIL14}\\
& \lim_{\dl\rightarrow 0} \sup_{|h|\le \dl} \frac{|\Im \dx^j u_s(x+h)-\Im \dx^j u_s(x)|}{\sqrt{2\s^2(|h|)\log\log \frac1{|h|}}}=0
\label{LIL15}
\end{align}

\noi
almost surely, for $j = 0, 1, \dots, r-1$.
 
In the following, we will take $r$ derivatives (in $x$) of both sides of \eqref{LIL4}
by setting  $y = x+h$.
In view of \eqref{LIL14} and \eqref{LIL15}, 
we see that,
after taking $r$ derivatives,  
dividing by $\sqrt{2\s^2(|h|)\log\log \frac1{|h|}}$,
and taking $\lim_{\dl\rightarrow 0} \sup_{|h|\le \dl}$, 
the only terms in \eqref{LIL4} that survive are those terms
where all the $r$ derivatives
falls only on $\Re u_s(x+h)-\Re u_s(x)$
(or $\Im u_s(x+h)-\Im u_s(x)$)
to which we can apply \eqref{LIL12} and \eqref{LIL13}.
Therefore, we obtain 
\begin{align*}
\lim_{\dl\rightarrow 0} \sup_{|h|\le \dl}
& \frac{\Re [\dx^r(e^{- i|u_s(x+h)|^2} u_s(x+h))] -\Re[\dx^r (e^{- i|u_s(x)|^2}u_s(x))]}
{\sqrt{2\s^2(|h|)\log\log \frac1{|h|}}}
 \notag\\
&\ge - 2\sin |u_s(x)|^2 (\Re u_s(x))^2\notag\\
&\phantom{X} - \big|\cos|u_s(x)|^2\big| -\big|\sin |u_s(x)|^2\big|
- 2\cdot \big|\sin |u_s(x)|^2 \Im u_s(x)\Re u_s(x)\big|\notag\\
&\phantom{X}
 - 2\cdot \big|\cos |u_s(x)|^2 \Im u_s(x)\Re u_s(x)\big|
- 2\cdot\big| \cos |u_s(x)|^2\big| (\Im u_s(x))^2,
\end{align*}

\noi
which is exactly the right-hand side of   \eqref{LIL5}.
The rest follows as in Subsection \ref{SUBSEC:5.1}.

\subsection{Critical case:\,$s\in \frac 12+\mathbb{N}$}

In this case, we cannot simply apply Proposition \ref{thm:LIL}  directly, 
because, taking $s=\frac 32$ for example, we have\footnote{This computation follows from 
the computations in the proof of Proposition 
\ref{PROP:LL} below.}  
\begin{align}
\s_{\frac 32}(x) = \mathbb{E}\big[|u_{\frac 32}(x)-u_{\frac 32}(0)|^2\big] 
&= 2\sum_{n\in \mathbb{Z}}\frac{|1-e^{inx}|^2}{(1+n^2)^{\frac 32}}  \notag\\
& = 2 \sum_{n=1}^\infty \frac{(1-\cos(nx))^2+\sin^2(nx)}{(1+n^2)^{\frac 32}} \nonumber \\
& \sim x^2 \cdot \log \frac{1}{|x|},  
\label{COV2}
\end{align}

\noi
as $x \to 0$.  
In particular, $\s_{\frac 32} (x)$ is not 
a normalized regularly varying function at zero with index 
$0< \al < 2$.
Hence, Proposition \ref{thm:LIL}  is not applicable.

In \cite{benassi}, the authors considered the Gaussian process on $\mathbb{R}^n$ with covariance function given by the kernel of the inverse of a quite general elliptic pseudodifferential operator and studied the precise regularity of the process. In particular, they obtained a result generalizing Proposition \ref{thm:LIL}  by very different methods from those in \cite{marcusrosen}. 

For us, the relevant operator is $2^{-1}(\text{Id}-\dx^2)^{s}$ on $\mathbb{T}$. 
In this case, which the authors of \cite{benassi} call ``critical'' owing to the behavior \eqref{COV2} of the increments, the relevant result from \cite[Theorem 1.3 (ii)]{benassi} reads as follows.

\begin{proposition}\label{prop: benassi}
Let $X_{\frac 32}$ be the stationary Gaussian process on $\mathbb{R}$ 
with the covariance operator  $2(\textup{Id}-\dx^2 )^{-\frac 32}$. 
Then,  $X_{\frac 32}(x)$ has continuous sample paths.
Moreover, there exists a constant $c_{\frac 32}>0$ such that for each $y\in \mathbb{R}$, we have
\[\limsup_{x\to y} \frac{|X_{\frac32}(x)-X_{\frac32}(y)|}{|x-y|\sqrt{\log \frac{1}{|x-y|} \log \log \log\frac{1}{|x-y|}}}=
c_{\frac{3}{2}}\]

\noi
almost surely.
\end{proposition}

The $\log\log$ from the classical law of the iterated logarithm and Proposition \ref{thm:LIL} 
 is now replaced by a factor involving the triply iterated logarithm $\log\log\log$. 
In the following, we state and prove an analogue of Proposition \ref{prop: benassi}
on $\T$
in a direct manner.  See Proposition~\ref{PROP:LL} below.
Using this almost sure constancy of the modulus of
continuity (Proposition~\ref{PROP:LL}), 
we can once again repeat the argument presented 
 in Subsection~\ref{SUBSEC:5.1}.

The results in \cite{benassi} are much more general than Proposition \ref{prop: benassi}. In particular, they apply to operators with variable coefficients. In that case, the local modulus of continuity of the process can change from point to point (although it is constant across different realizations of the sample path). 
In our specific case, it is possible to give a more elementary proof,
 using the classical Khintchine's law of the iterated logarithm 
for independent sums, that 
the process $u_{\frac 32}$ has an exact modulus of continuity almost surely.
In terms of the setting in \cite{benassi}, 
this simplified proof comes as no surprise since
our process $u_{\frac 32}$ has a particularly simple representation as a sum of independent terms with respect to which the covariance operator is diagonal.

\begin{proposition}\label{PROP:LL}
Let  $u_{\frac 32}$ be given by  the random Fourier series in  \eqref{fBM}
with $s = \frac 32$. 
Then, for each $y\in \mathbb{T}$, we have
\begin{align}
\limsup_{x\to y} \frac{|u_{\frac32}(x)-u_{\frac 32}(y)|}
{2^{\frac 32}|x-y|\sqrt{\log \frac{1}{|x-y|} \log\log\log \frac{1}{|x-y|}}}=1
\label{LIL16a}
\end{align}

\noi
almost surely.
\end{proposition}

Once we prove Proposition \ref{PROP:LL}, we can proceed as in Subsection \ref{SUBSEC:5.1}
when $s = \frac 32$.
For $s \in \frac 32 + \N$, the modification is straightforward following the second half of Subsection 
\ref{SUBSEC:5.2} and thus we omit details.

\begin{proof}
Without loss of generality, set $y = 0$.
By writing 
\begin{align}
u_{\frac32}(x)-u_{\frac32}(0)
&= \sum_{n\in \Z\setminus\{0\}}^\infty \frac{e^{inx}-1}{\jb{n}^\frac{3}{2}}g_n \notag \\
&= \sum_{n=1}^\infty \frac{\cos(nx)-1}{\jb{n}^\frac{3}{2}} (g_n+g_{-n})
 +i \sum_{n=1}^\infty \frac{ \sin(nx)}{\jb{n}^\frac{3}{2}}(g_n-g_{-n}),
 \label{LIL16}
\end{align}

\noi
we first show that the first term on the right-hand side of \eqref{LIL16} does not contribute
to the limit in \eqref{LIL16a}.
Then, we break up the second sum into $\log \frac{1}{|x|}$ pieces, each with variance of order 1, plus a small remainder, 
and then  apply the classical law of the iterated logarithm for a sum of i.i.d.~random variables. 
As we see below, 
the leading order contribution comes from the sum
\begin{equation}\label{sineterm}
\sum_{n=1}^{\lfloor \frac{1}{|x|}\rfloor }
\frac{\sin (nx)}{\jb{n}^\frac{3}{2}}(g_n-g_{-n}).
\end{equation}

We split the first sum on the right-hand side of  \eqref{LIL16} into 
$\big\{n > \lfloor \frac{1}{|x|}\rfloor \big\}$
and $\big\{1 \leq n \leq \lfloor \frac{1}{|x|}\rfloor \big\}$.
The contribution from $\big\{n > \lfloor\frac{1}{|x|}\rfloor\big\}$ is a mean-zero Gaussian random variable with variance
\begin{align}
 \s_L^2 := 4\sum_{n=L}^\infty \frac{(\cos(nx)-1)^2}{\jb{n}^3} \les L^{-2}.
\label{V1}
 \end{align}

\noi
In particular, when $L = \lfloor\frac{1}{|x|}\rfloor + 1$, 
we have $\s_L^2 = O(x^2)$.
Then, 
for $\lambda>0$, we have
\[P\left(\bigg|\sum_{n=L}^\infty \frac{\cos (nx)-1}{\jb{n}^\frac{3}{2}}(g_n+g_{-n})\bigg| 
\ge  \s_L \ld \right)\les e^{-c\ld^2}. \]

\noi
Taking $\ld= c \sqrt{\log L}$ for sufficiently large $c>1$, 
the right-hand side is summable in $L$.
Hence, by  the Borel-Cantelli lemma and the variance bound \eqref{V1},
there exists  $C>0$ such that
\[\sup_{L\ge L_0} \frac{\left|\sum^\infty_{n=L} 
\frac{\cos (nx)-1}{\jb{n}^\frac{3}{2}}(g_n+g_{-n})\right|}{CL^{-1}\sqrt{\log L}}\le 1 \]

\noi
for some $L_0 = L_0(\o) < \infty$ 
 with probability 1. 
 As a consequence, we obtain
\[\limsup_{x\to 0} \frac{\left|\sum_{n>\lfloor \frac{1}{|x|}\rfloor }
 \frac{\cos(nx)-1}{\jb{n}^\frac{3}{2}}(g_n+g_{-n})\right|}{|x|\sqrt{\log \frac{1}{|x|}\log\log\log \frac{1}{|x|}}}=0\]
almost surely.

The contribution from $\{ 1\leq n \leq \lfloor \frac{1}{|x|}\rfloor \}$ to the first sum 
on the right-hand side of  \eqref{LIL16} 
can be estimated in a similar manner by noticing
that it is a mean-zero Gaussian random variable
with variance
\[ \sum_{n=1}^{\lfloor \frac{1}{|x|}\rfloor}  \frac{(\cos(nx)-1)^2}{\jb{n}^3}
\les x^4 \sum_{n=1}^{\lfloor \frac{1}{|x|}\rfloor}  \frac{n^4}{\jb{n}^3}\les x^2. \]

\noi
This shows that the contribution from the first sum on the right-hand side of  \eqref{LIL16}
to the limit \eqref{LIL16a} is 0.

Next, we consider the second sum on the right-hand side of \eqref{LIL16}.
The contribution from $\{  n > \lfloor \frac{1}{|x|}\rfloor \}$ 
can be estimated as above.
We split the main term in \eqref{sineterm} as follows.
Write
\begin{equation}\label{split}
\sum_{n=1}^{\lfloor \frac{1}{|x|}\rfloor }\frac{\sin (nx)}{\jb{n}^\frac{3}{2}}
 (g_n-g_{-n})
= x\sum_{n=1}^{\lfloor \frac{1}{|x|}\rfloor }\frac{n}{\jb{n}^\frac{3}{2}}(g_n-g_{-n})
+\sum_{n=1}^{\lfloor \frac{1}{|x|}\rfloor}\frac{h(nx)}{\jb{n}^\frac{3}{2}}(g_n-g_{-n}),
\end{equation}

\noi
where 
\[h(z) = \sin z - z \sim z^3(1+o(1)).\]

\noi
The second term in \eqref{split} can be treated as a remainder
by noticing that that it is a mean-zero Gaussian random variable
with variance
\begin{align*}
4\sum_{n=1}^{\lfloor \frac{1}{|x|}\rfloor } \frac{h^2(nx )}{\jb{n}^3} 
\les  x^6\sum_{n=1}^{\lfloor \frac{1}{|x|}\rfloor }\frac{n^6}{\jb{n}^3}
\les x^2.
\end{align*}

It remains to consider the first term in \eqref{split}. 
First, define a sequence $\{N(k)\}_{k = 0}^\infty \subset \N$
by setting $N(0) = 0$ and 
\[N(k) = \min\bigg\{N> N(k-1): \sum_{n=N(k-1)+1}^{N}\frac{n^2}{\jb{n}^3} \geq 1\bigg\}\]

\noi
for $k \in \N$.
Noting that 
\begin{align}
\sum_{n=M}^N \frac{n^2}{\jb{n}^3}
= \sum_{n=M}^N \frac{1}{n} + O\bigg(\sum_{n = M}^N \frac{1}{n^3}\bigg) 
= \log N - \log M +O\left(\frac{1}{M}\right)
\label{LIL18}
\end{align}

\noi 
for $N > M \geq 1$, 
we first see that  $N(k) \ge C_1 e^k$.
Using \eqref{LIL18} once again,
\begin{align*}
\log N(k) + O(1)= 
\sum_{n=1}^{N(k)} \frac{n^2}{\jb{n}^3}
\leq k +  \sum_{n=1}^k\frac{1}{N(n)+1}
\leq k + O(1), 
\end{align*}

\noi
giving $N(k) \leq C_2 e^k$.
Putting together, we have
 \begin{align}
C_1 e^k \leq N(k) \leq C_2 e^k.
\label{LIL19}
 \end{align}

Now, we define a sequence $\{X_{k}\}_{k\in \N}$ of independent Gaussian random variables by setting 
\[X_k = x\sum_{n=N(k-1)+1}^{N(k)}\frac{n}{\jb{n}^\frac{3}{2}}(g_n-g_{-n}). \]

\noi
Then, we have 
\[\mathbb{E}\big[|\Re X_k|^2\big] =
\mathbb{E}\big[|\Im X_k|^2\big] =  x^2\big(1+ O(N(k)^{-1})\big).\]

Finally, 
define $L(|x|)$ by 
\[L(|x|)=\inf \left\{ k: N(k) \ge \Big\lfloor \frac{1}{|x|}\Big\rfloor \right\}.\]

\noi
Then, from \eqref{LIL19}, we have 
\[L(|x|)=\log \frac{1}{|x|}\cdot (1+o(1)).\]

\noi
Applying  Khintchine's law of the iterated logarithm to the sum
\[S(x) = \sum_{k=1}^{L(|x|)} X_k,\]

\noi
we find
\[\limsup_{x\to  0} \frac{S(x)}{2^{\frac 32}|x|\sqrt{\log\frac{1}{|x|}\log\log\log{\frac{1}{|x|}}}}=1\]
almost surely.
This completes the proof of Proposition \ref{PROP:LL}.
\end{proof}

\begin{ackno}\rm
T.O.~was supported by the European Research Council (grant no.~637995 ``ProbDynDispEq'').
N.T.~was supported by the European Research Council (grant no.~257293 ``DISPEQ'').
The authors are also grateful to the anonymous referees for their helpful comments that have improved the presentation of this paper.
\end{ackno}

\end{document}